\documentclass[10pt,twocolumn,twoside]{IEEEtran}                       

\IEEEoverridecommandlockouts                              
\usepackage{enumerate}
\usepackage{color,soul,arydshln}
\usepackage{url}
\usepackage{subfigure}
\usepackage{epstopdf}
\usepackage{float}
\usepackage{amsmath,  amssymb, bbm, amsthm}
\usepackage{algorithmic, algorithm}
\usepackage{url, color}
\usepackage{graphicx} 

\newtheorem{theorem}{Theorem}

\newtheorem{lemma}[theorem]{Lemma}

\newtheorem{assumption}{Assumption}
\newtheorem{definition}{Definition}

\newcommand{\hide}[1]{}

\usepackage{ifthen}
\newboolean{showcomments}
\setboolean{showcomments}{true}
\newcommand{\yue}[1]{\ifthenelse{\boolean{showcomments}}
{ \textcolor{red}{(Yue says:  #1)}}{}}
\newcommand{\josh}[1]{\ifthenelse{\boolean{showcomments}}
{ \textcolor{red}{(Josh says:  #1)}}{}}
\newcommand{\zhenhua}[1]{\ifthenelse{\boolean{showcomments}}
{ \textcolor{red}{(Zhenhua says:  #1)}}{}}
\newcommand{\nhattan}[1]{  \ifthenelse{\boolean{showcomments}}
{ \textcolor{red}{(Tan says:  #1)}} {}  }
\newcommand{\david}[1]{  \ifthenelse{\boolean{showcomments}}
{ \textcolor{red}{(David says:  #1)}} {}  }
\newcommand{\addcites}[0]{\ifthenelse{\boolean{showcomments}}
{ \textcolor{green}{(add citation(s))}}{}}
\newcommand{\addcite}[0]{\ifthenelse{\boolean{showcomments}}
{ \textcolor{green}{(add citation(s))}}{}}
\newcommand{\addref}[0]{\ifthenelse{\boolean{showcomments}}
{ \textcolor{green}{(add ref)}}{}}
\newcommand{\todo}[1]{\ifthenelse{\boolean{showcomments}}
{ \textcolor{red}{(To do:  #1)}}{}}
\newcommand{\revjournal}[1]{\ifthenelse{\boolean{showcomments}}
{\textcolor{black}{#1}}{}}
\newcommand{\fixes}[1]{\ifthenelse{\boolean{showcomments}}
{\textcolor{black}{#1}}{}}

\newboolean{shortversion}
\setboolean{shortversion}{false}
\newcommand{\sversion}[1]{\ifthenelse{\boolean{shortversion}}
	{#1}{}}

\newboolean{longversion}
\setboolean{longversion}{true}
\newcommand{\lversion}[1]{\ifthenelse{\boolean{longversion}}
	{#1}{}}

\newboolean{otherproof}
\setboolean{otherproof}{true}
\newcommand{\oproof}[1]{\ifthenelse{\boolean{otherproof}}
	{#1}{}}

\graphicspath{{images/}}

\title{
	Geographically Coordinated Primary Frequency Control Technical Report
}

\author{Joshua Comden, Tan N. Le, Yue Zhao, Bong Jun Choi, and Zhenhua Liu*\\ 
	\thanks{J. Comden, T. Le, Y. Zhao, and Z. Liu are with the College of Engineering and Applied Sciences, Stony Brook University, 100 Nicolls Road, Stony Brook, NY 11794, USA. Email: \{joshua.comden, tan.le, yue.zhao.2, zhenhua.liu\}@stonybrook.edu.}
	\thanks{B. Choi is with the School of Computer Science and Engineering \& School of Electronic Engineering, Soongsil University, 369 Sang-doro, Sangdo-dong, Dongjak-gu, Seoul, Korea. Email: davidchoi@soongsil.ac.kr.}
	\thanks{This research is supported by NSF grants CNS-1464388, CNS-1617698, CNS-1730128, CNS-1717558, DGE-1633299, and was partially funded by MSIP, Korea, grant NRF-2019R1C1C1007277. We would like to thank Dr. Changhong Zhao for his valuable and insightful comments on this paper and the reviewers for their suggestions.}
	\thanks{*Corresponding author.}
}

\begin{document}

\maketitle
\thispagestyle{empty}
\pagestyle{empty}


\begin{abstract}
Primary Frequency Control (PFC) is a fast acting mechanism used to ensure high-quality power for the grid that is becoming an increasingly attractive option for load participation.
Due to speed requirement and other considerations,
\revjournal{it is often desirable to have}
\emph{distributed} control laws.
Current
\revjournal{distributed}
PFC designs assume that the costs at each geographic location are \emph{independent}.
However, many networked systems, such as those for cloud computing, 
\revjournal{have \emph{interdependent} costs across locations}
and therefore need geographic coordination.
In this paper, distributed control laws are designed for interdependent, geo-distributed loads in PFC based on the optimality conditions of the global system.
The controlled frequencies are provably stable, and the final equilibrium point is proven to strike an optimal balance between load participation and the frequency's deviation from its nominal set point.
We evaluate the proposed control laws with realistic numerical simulations.
\fixes{Under current technology, the proposed control laws achieve a convergence time that is smaller than droop control alone and is comparable to that of distributed control without interdependent costs.}
Results also highlight significant cost savings over existing approaches under a variety of settings.
\end{abstract}

\section{INTRODUCTION}


In the electric power grid, keeping a stable frequency at a set nominal value is important for supplying reliable high-quality power and maintaining safe grid infrastructure operation.
The frequency can drift away from its set point if there is a power imbalance anywhere in the grid.
To stabilize and return the frequency back to its nominal value, Frequency Control (FC) \cite{shayeghi2009load} is used to correct this power imbalance. 
FC as a whole involves different mechanisms working at a range of timescales. 
In the presence of a sudden power imbalance, e.g., generator failure, Primary Frequency Control (PFC) \cite{molina2011decentralized} is used to stop the drift and stabilize the frequency within tens of seconds. This is typically done by independently controlling each generator's power injection according to a function of its \emph{locally} measured deviated frequency. 
Usually this stabilized equilibrium frequency is deviated from its nominal value, and Automatic Generation Control (AGC) \cite{jaleeli1992understanding} is then used to \emph{centrally} decide and change the set power injections of the generators for bringing frequency back to nominal within a few minutes. 


Traditionally, FC is done on the generation side, while it is starting to become an attractive opportunity for load participation.
Similar to generation-side, demand-side PFC works by setting devices to independently adjust their individual power consumptions according to some function of the locally measured frequency.
However when a device deviates its power consumption, there is an associated loss of utility to the owner of that device.
Finding an optimal balance between the cost of load participation and stabilizing frequency 
is a major challenge of demand-side PFC.


For large-scale, geographically distributed system,
\fixes{a centralized control law would need to collect all the locally measured information at all times.}
Additionally, privacy requirements may not allow a central entity to know the objectives and constraints of all the users.
This is especially true for deregulated power markets and load frequency control when some coordination is needed but local objectives cannot be shared to a central authority~\cite{camponogara2002distributed}.
Therefore, it is often desirable to have distributed control laws, which is the focus of this work. 
It has been shown by \cite{zhao2014optimal} that with well designed control laws, this optimal balance can be made in a fully decentralized fashion and is provably stable as long as these costs are assumed to be \emph{independent} of each other.
Intuitively, this \fixes{above} work leverages locally measured frequency deviations to infer power system status and adjust local power consumptions independently without any communication. 
Recently, this was extended to non-linear power flows with moderate communication \cite{zhao2017distributed}.

\begin{figure}[t]
	\centering
	\includegraphics[width=1.0\linewidth]{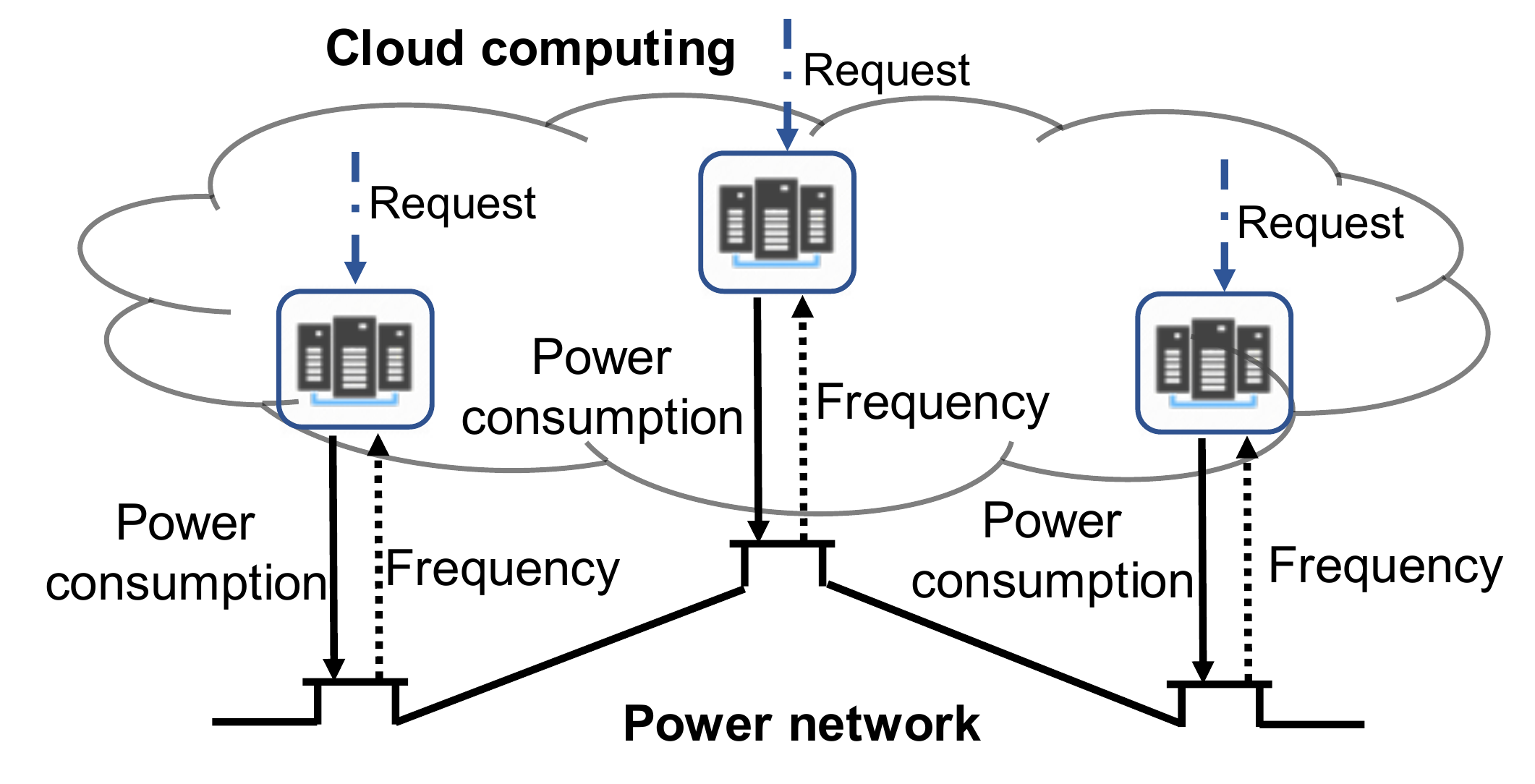}
	\caption{Cloud participation in Primary Frequency Control:
	Multiple datacenters connected to the power network at different locations serve distributed workload requests and adjust their loads according to local frequency measurements.
	The requests may be rerouted to different datacenters by the global load balancer.
	Therefore, the datacenters' power consumptions are not independent.}
	\label{fig:system_model}
	\vspace{-0.05in}
\end{figure}


However, the assumption of independent costs is somehow restrictive. For instance, in some networked systems such as in cloud computing, not all costs can be considered independent.
In general, networks of datacenters make up the infrastructure that supplies the computing resources necessary for making the cloud run.
These datacenters are located around the world and groups of them may be connected to the same electrical power network.
User workloads requesting computing resources are distributed to different datacenters in a way that depends on data availability, server utilization, network delay, etc.
Since servers essentially convert electrical power into computational power, the distribution of IT workloads among the datacenters has a direct impact on the distribution of their power consumptions.
Through Geographic Load Balancing (GLB), networks of datacenters can dynamically redistribute workloads depending on datacenter and power network conditions~\cite{liu2011greening}. 
In other words, some workloads that cannot be processed in one datacenter can be served by other datacenters, which significantly increases system reliability and flexibility in workload distribution. 
However, if some fraction of the workload is not processed by \emph{any} of the datacenters, the whole system is penalized from the resulting loss of revenue \cite{antonescu2013dynamic}.


Since the cloud has been steadily increasing its share of the total US electricity consumption to about 1.8\% in 2014 and has precise energy management of its systems \cite{shehabi2016united}, it has great potential to be a large contributor to PFC.
\revjournal{In fact, a datacenter can control the energy consumption of its servers at the granularity of tens of milliseconds or faster \cite{meisner2009powernap}, and can communicate between other datacenters with only milliseconds of delay \cite{moniz2017blotter}}.
\revjournal{The system's overall architecture is shown in Figure \ref{fig:system_model}.}
For this reason, the paper aims at the following question: \textbf{How to coordinate primary frequency control with geographic interdependent costs in a distributed manner?} While motivated by cloud computing, this question and associated solutions can be applied to general cases with interdependent costs.




We make the following contributions in this paper:
\begin{enumerate}
	\item We formulate a primary frequency control problem that balances the extent of load control participation for a cloud computing network operating at different geographic locations with both independent and \emph{interdependent} costs (Sections \ref{sec:model} and \ref{sec:problem}).
	
	\item We study the frequency control problem's optimal solution characteristics, and design a set of distributed feedback control laws
	\revjournal{inspired by the Subgradient Method from convex optimization}
	that a) has an optimal equilibrium point, and b) is asymptotically stable
	(Sections \ref{sec:optima} and \ref{sec:algorithm}).

	\item \fixes{We demonstrate our control scheme on a realistic emulator, Power System Toolbox (PST)~\cite{chow2009power}, and show that it can both help stabilize the power system faster than droop control alone and achieve an equilibrium frequency that is closer to the nominal value.}
	Furthermore, we show that our distributed control gives significant cost savings as interdependent costs become more prevalent (Section \ref{sec:perf}).
\end{enumerate}

Section \ref{sec:motiv} gives a motivating example to show the impact of interdependent costs on a cloud computing network participating in PFC.
\fixes{Our preliminary work was presented in \cite{comden2017geographically}, whereas in this paper we
incorporate droop control and evaluate the effects of communication time delays which were not present in the previous publication.}
\section{RELATED WORK}




Controlling the frequency of a power system and the economics of its control mechanism has been a major topic of interest for many decades \cite{schweppe1980homeostatic}.
The goal is to both stabilize the frequency and bring it back to the desired set point.
Since there are different types of frequency control categorized by resource response times, various methods of coordination between the resources' control actions have been developed, which include hierarchical control \cite{kiani2012hierarchical}, layered control \cite{ilic2007hierarchical}, and distributed control \cite{dorfler2016breaking}.



Primary frequency control being the fastest, requires the use of governor-type actions that include droop control for stabilizing the frequency even if it is not at the desired nominal value \cite{schweppe1980homeostatic}.
Complex and stochastic models allow more complicated control mechanisms such as droop control with a dead band \cite{chertkov2017chance}.
\lversion{With the increasing penetration of wind power, control of wind turbine speed has potential to be used as an additional PFC resource \cite{margaris2012frequency}.}
This type of control can be made to be decentralized so that it can be implemented with only local feedback \cite{lu1989nonlinear,trip2016internal}.
\lversion{For more operational passive control instead of sudden disturbances, \cite{trip2016internal} designs internal-model controllers to work for time varying loads.}
\fixes{The recent theoretical work of \cite{cherukuri2018role} opens up an avenue for new saddle-point controllers derived from convex optimization problems that have provable asymptotic stability.}


In addition to the generation-side of frequency control, there has been active work towards incorporating demand-side participation to help with both stability and cost-effectiveness of operating the power network \cite{trudnowski2006power}.
\lversion{This includes smart appliances that have the ability to change their power consumption under different frequency measurements.
Simulations have shown considerable potential cost and added stability benefits \cite{lu2006design,short2007stabilization}.}
Datacenters have recently been targeted for demand response participation to help balance supply and demand in the power grid \cite{ghatikar2012demand}.
\cite{zhou2015smart} developed an auction mechanism to incentivize participation in demand response from a geo-distributed cloud provider.
For datacenters that have multiple tenants, \cite{chen2015greening,zhang2015truthful} design a pricing mechanism to extract load reductions during emergency periods for the grid.
\section{MOTIVATING EXAMPLE}
\label{sec:motiv}

In order to demonstrate the importance of taking geographic interdependence into account for PFC design, we showcase a concrete example of a cloud computing workload running on a network of datacenters which consume electrical power from the power grid.


Consider a network of two identical datacenters where the only difference between them is their locations in the power grid and their efficiencies which are the ratios of computational power output to the electrical power input.
The first datacenter is a high efficiency one with an efficiency of 0.9 while the second is an average one with an efficiency of 0.5.
The first is 1.8$\times$ more efficient than the second which is consistent with those studied in~\cite{cheung2014energy,shehabi2016united}.
Therefore, the total computational power of the datacenters is a linear combination of both power consumptions $(d_1,d_2)$ weighted by their efficiencies.
The workload that the datacenters need to share requires 28 MW of computing power as measured by an ideal fully efficient datacenter.
Subtracting the workload size from the total computational power gives the excess computational power of the network $0.9d_1+0.5d_2-28$.
A negative excess computational power means that some of the workload was not processed which results in a loss of revenue shared by all of the datacenters.
\fixes{This cost is typically super linear with respect to the computational power since low value jobs are delayed or dropped first before higher value jobs.
In this example it is represented by the one-sided quadratic cost function $\gamma((28-0.9d_1-0.5d_2)^+)^2$ that features an increasing marginal cost when computational power goes below 28 MW where $\gamma$ is the \emph{interdependent} cost coefficient and $(x)^+:=\max\{0,x\}$.
Each datacenter also observes an \emph{independent} cost for purchasing and processing $d_j$ of electrical power into computational power.
In this example, we use the quadratic function $(d_j-20)^2$ which corresponds to the cost of deviating away from its normal operational set point of 20 MW.
The quadratic cost is a common model for disutility within the electricity markets literature~\cite{allaz1993cournot,cai2013inefficiency,murphy2010impact,yao2007two}.
However, our work in this paper handles more general strictly convex cost functions.
In this example,} the total cost of both datacenters is:
$\fixes{\gamma\left((28-0.9d_1-0.5d_2)^+\right)^2} + (d_1-20)^2 + (d_2-20)^2.$
The solution $d_1=d_2=20$ MW minimizes the total cost with a value of 0 and gives an aggregate electrical power consumption of 40 MW.
Let this be the power consumed when the datacenters are \fixes{operating normally and} not participating in FC.

\begin{figure}[!t]
	\begin{center}
		\subfigure[Interdependent Costs]{{\includegraphics[width=0.45\columnwidth]{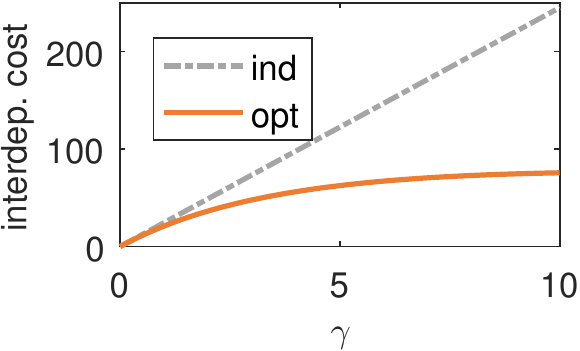}}
			\label{fig:motiv_inter}}
		\subfigure[Independent Costs]{{\includegraphics[width=0.45\columnwidth]{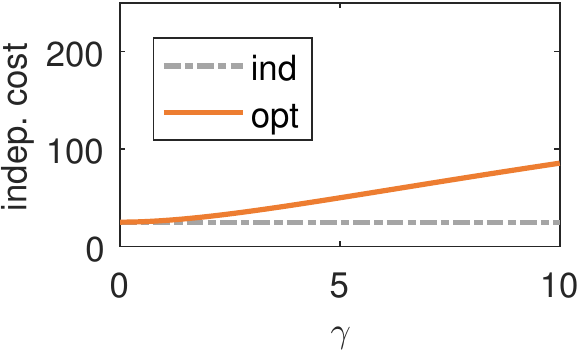}}
			\label{fig:motiv_indep}}
		\vspace{-0.05in}		
		\caption{Interdependent (a) and Independent (b) costs vs. the interdependent cost coefficient $\gamma$ for FC that only considers independent costs (dashed gray line) and of the optimal FC (solid orange line) .}
		\label{fig:motivation}
	\end{center}
	\vspace{-0.15in}
\end{figure}


Suppose that there is a sudden power disturbance in the grid and therefore it is required that the datacenters reduce their aggregate power consumption by 10 MW for FC, i.e. $\Delta d_1+ \Delta d_2=-10$.
If only the independent costs (last two terms of the total cost) are considered for FC, then reducing both power consumptions by 5 MW minimizes the cost.
However, if the interdependent cost (first term of the total cost) is also considered then reducing both power consumptions by 5 MW may be suboptimal as shown in Figure \ref{fig:motivation}.
Figure \ref{fig:motiv_inter} shows that interdependent cost for FC that only minimizes the independent cost increases linearly with the interdependent cost coefficient $\gamma$ while optimal FC increases sublinearly.
This gap is not fully compensated by the difference in independent costs (See Figure \ref{fig:motiv_indep}).
Specifically, when the interdependent cost coefficient $\gamma$ is large, then the total cost for only minimizing independent costs is much higher than that of optimal (e.g. 65\% \fixes{higher total cost} when $\gamma=10$).
Therefore it is important for FC to take into account interdependent costs for networked systems and the rest of this paper focuses on closing this gap. 
As a result, there is great opportunity to reduce the total cost by taking the interdependent costs into the optimization.

\begin{table}
	\caption{Important Model Notations}
	\label{tbl:notation}
	\begin{center}
		\begin{tabular}{| r c l |}
			\hline
			\textbf{Power Network} & $\mathcal{N},\mathcal{E}$ & Set of buses and directed lines \\
			& $\omega_0$ & Nominal frequency \\
			\hdashline[1pt/3pt]
			Bus $j\in\mathcal{N}$: & $\theta_j$ & Voltage phase angle deviation \\
			& $\omega_j$ & Frequency deviation \\
			& $P_j$ & Total power injection \\
			& $F_j(\cdot)$ & \fixes{Net} power outflow \\
			& $D_j$ & Frequency-sensitivity coefficient \\
			& \fixes{$R_j$} & \fixes{Droop control value} \\
			& $p_j$ & Constant power injection \\
			& $M_j$ & Physical inertia \\
			& $|V_j|$ & Voltage magnitude \\
			\hdashline[1pt/3pt]
			Line $(j,k)\in\mathcal{E}$: & $x_{jk}$ & Reactance \\
			& $Y_{jk}$ & Maximum power flow \\
			\hline
			\textbf{Cloud Computing} & $\mathcal{D}$ & Set of datacenters  \\
			& $W$ & Incoming workload rate \\
			& $s$ & Excess processing power \\
			& $g(\cdot)$ & Interdependent cost function \\
			\hdashline[1pt/3pt]
			Datacenter $j\in\mathcal{D}$: & $r_j$ & Processing power \\
			& $d_j$ & Controllable load \\
			& $c_j(\cdot)$ & Independent cost function \\
			& $\underline{d}_j,\overline{d}_j$ & Min/max power usage \\
			& $a_j$ & Computational efficiency \\
			\hline
		\end{tabular}
	\end{center}
\end{table}



\section{MODEL AND NOTATION}
\label{sec:model}

\subsection{Power network model}


We consider a power network consisting of a set of buses $\mathcal{N}$ connected by a set of 
\fixes{directed}
lossless lines $\mathcal{E}$, and only consider the real power injection at each bus and the real power flow across each line.
Important model notation can be found in Table \ref{tbl:notation}.
We ignore reactive power since it is more closely related to voltage control as compared to frequency control \cite{rebours2007survey}.
For each bus $j\in\mathcal{N}$ we denote $P_j$ as the real power injection, $\theta_j$ as the voltage phase angle from a standard reference point rotating at the set nominal frequency $\omega_0$, $\omega_j$ as the frequency's \emph{deviation} from the nominal set point $\omega_0$ or also the time rate of change for the voltage phase angle
\vspace{-0.05in}
\begin{align}
	\omega_j := \frac{d\theta_j}{dt} \quad \forall j\in\mathcal{N}, \label{eq:FreqDev}
\end{align}
and we assume that the voltage magnitude $|V_j|$ remains constant during the time frame of PFC.
For each 
line $(j,k)\in\mathcal{E}$ we denote $x_{jk}$ as the reactance.


The power injection at each bus is split into \fixes{four}
terms
\vspace{-0.05in}
\begin{align}
	P_j := p_j - D_j\omega_j - d_j \fixes{-\frac{1}{R_j}\omega_j} \quad \forall j\in\mathcal{N} \label{eq:pj_def}
\end{align}
where $p_j$ is the frequency-insensitive part of the non-controllable power injection, the second term is the frequency-sensitive part of the non-controllable power injection, $d_j$ is the controllable load, and
\fixes{the last term represents the droop control with $R_j$ as the droop control setting.
If $j\notin\mathcal{D}$, then $d_j=0$.}
We assume that $p_j$ remains constant during the time frame of PFC.
The second term approximates the frequency-sensitivity
\revjournal{from its nominal value}
as a first-order dependence on the frequency's deviation where $D_j>0$ is the linear coefficient.
This is reasonable for small deviations \cite{bergen1981structure}.


The voltage phase angle difference across each line connected to bus $j\in\mathcal{N}$ determines the
\fixes{net}
real power flow out from that bus into the rest of the power network:
\vspace{-0.05in}
\begin{align}
	F_j(\boldsymbol{\theta}):= \sum_{k:(j,k)\in \mathcal{E}}Y_{jk}\sin(\theta_j-\theta_k) \nonumber \\
	- \sum_{i:(i,j)\in \mathcal{E}}Y_{ij}\sin(\theta_i-\theta_j) \label{eq:NetFlow}
\end{align}
where $Y_{jk}:=\frac{|V_j||V_k|}{x_{jk}}$ is the maximum power flow across line $(j,k)\in\mathcal{E}$ as determined by the constant bus voltage magnitudes and line reactance.

\subsection{Cloud computing model}


We consider a set of datacenters $\mathcal{D}\subseteq\mathcal{N}$ connected by a high speed communication network that provides cloud computing services for a workload incoming rate of size $W$.
We assume that the incoming workload rate remains constant for a primary frequency control event time duration.
Each datacenter $j\in\mathcal{D}$ processes some of the incoming workload at a rate of $r_j$.
Subtracting the workload incoming rate from the sum of the datacenter processing rates we get the excess computational power of the network:
\vspace{-0.05in}
\begin{align}
	s := \sum_{j\in\mathcal{D}}r_j - W. \label{eq:s_of_r}
\end{align}
A negative $s$ represents insufficient computational power to process all of the incoming workload which means that $-s$ of the workload may suffer a delay or remain unprocessed.
This causes a loss of revenue captured by the cost function $g(s)$ which means that a specific cost level \emph{depends on all} of the datacenter processing rates.
We assume that each datacenter knows the interdependent cost function $g(\cdot)$ and can receive information about the excess computational power $s$ via the communication network instantaneously.
\revjournal{This assumption is reasonable since PFC is targeted for convergence within tens of seconds while communication between datacenters is on the order of milliseconds \cite{moniz2017blotter}.}


We model each datacenter $j$ as a machine that converts electrical power $d_j$ into computational power $r_j$ with a linear usage profile:
\vspace{-0.1in}
\begin{align}
	d_j := \underline{d}_j + \frac{1}{a_j}r_j \label{eq:dj_defn}
\end{align}
where $\underline{d}_j$ is the constant overhead electrical power usage, and $a_j$ is the conversion coefficient that can be considered the computational efficiency as described in Chapter 5 of \cite{barroso2013datacenter}.
\revjournal{We assume that a datacenter can change its power demand instantaneously since a server can control its power consumption within tens of milliseconds \cite{meisner2009powernap}}. 
Each datacenter has an upper bound on its electrical power consumption $\overline{d}_j$, and observes a cost of $c_j(d_j)$ associated with obtaining and processing the electrical power $d_j$.
We assume that the value of $d_j$ and the function $c_j(\cdot)$ are only known by datacenter $j$.

The excess computational power of the network \eqref{eq:s_of_r} can now be expressed in terms of the individual electrical power consumptions in \eqref{eq:dj_defn}:
\vspace{-0.1in}
\begin{align}
	s = \sum_{j\in\mathcal{D}}a_jd_j - b \label{eq:s_of_djs}
\end{align}
where $b:=W+\sum_{j\in\mathcal{D}}a_j\underline{d}_j$.

\subsection{Power System Dynamics}


The power network frequency dynamics at each bus are determined by the swing equation
\vspace{-0.05in}
\begin{align}
	M_j\frac{d\omega_j}{dt} = P_j - F_j(\boldsymbol{\theta}),  \quad \forall j\in\mathcal{N} \label{eq:Swing_full2}
\end{align}
where $M_j$ is the physical inertia of the rotating equipment.


Since stability is an essential feature of FC, we give the following definition for an equilibrium point of the 
system we study.
\vspace{-0.05in}
\begin{definition}\label{def:eqlbrm}
	A closed-loop equilibrium of the system \eqref{eq:FreqDev} \eqref{eq:pj_def} \eqref{eq:NetFlow} \eqref{eq:s_of_djs} \eqref{eq:Swing_full2}, is any solution $(\boldsymbol{\theta}^*,\boldsymbol{\omega}^*,\mathbf{P}^*,\mathbf{d}^*,s^*)$ that further satisfies:
	\vspace{-0.1in}
	\begin{subequations}
		\begin{align}
			\frac{d\omega_j^*}{dt} & = 0 & \quad \forall j\in\mathcal{N} \label{eq:eqlbrm_omega_constant} \\
			\frac{dP_j^*}{dt} & = 0 & \quad \forall j\in\mathcal{N} \label{eq:eqlbrm_pj_constant}\\
			\omega_j^* & = \omega^* & \quad \forall j\in\mathcal{N}. \label{eq:eqlbrm_eq_omega}
		\end{align}
	\end{subequations}
\end{definition}

Note that \eqref{eq:eqlbrm_omega_constant} makes the LHS of \eqref{eq:Swing_full2} equal to zero and \eqref{eq:eqlbrm_eq_omega} synchronizes all deviations of frequency to a single value.  \eqref{eq:eqlbrm_pj_constant} implies that $d(d_j^*)/dt=0:\forall j\in\mathcal{N}$ and thus $ds^*/dt=0$.
\section{GEOGRAPHIC FREQUENCY CONTROL PROBLEM}
\label{sec:problem}

Before designing control laws for datacenter participation in PFC, we must first decide what an optimal balance is between the controllable datacenter loads and equilibrium deviations of frequency.
We form an optimization problem that minimizes the global cost of the system which is the total cost of the network of datacenters plus the summed cost of the equilibrium deviations of frequency at each bus.


The total cost of the network of datacenters is the interdependent cost of the excess computational power summed with the independent electrical power costs:
\vspace{-0.05in}
\begin{align}
	g(s) + \sum_{j\in\mathcal{D}}c_j(d_j). \nonumber
\end{align}
For the cost of the deviations of frequency, we adopt the cost function developed by \cite{zhao2014design} which is a sum of the squared deviations at each bus weighted by its associated frequency-sensitive linear coefficient.
\fixes{Since we are interested in the cost of the deviations remaining after standard PFC, we include the droop control values with the frequency-sensitive linear coefficients in its cost weightings:}
\vspace{-0.05in}
\begin{align}
	\sum_{j\in\mathcal{N}}\frac{1}{2}\left(D_j\fixes{+\frac{1}{R_j}}\right)\omega_j^2 \nonumber \label{eq:fj_quad}
\end{align}


Given the above total cost of the network of datacenters and equilibrium deviations of frequency, we
\revjournal{give the steady-state}
Geographic Frequency Control (GFC) problem:
\vspace{-0.05in}
\begin{subequations}\label{eq:gfc}
	\begin{align}
		\min_{s,\mathbf{d},\boldsymbol{\omega}} \quad & g(s) + \sum_{j\in\fixes{\mathcal{D}}}c_j(d_j) + \fixes{\sum_{j\in\mathcal{N}}}\frac{1}{2}\left(D_j\fixes{+\frac{1}{R_j}}\right)\omega_j^2 \\
		\text{s.t.} \quad & \sum_{j\in\fixes{\mathcal{D}}}a_jd_j - b = s \label{eq:gfc_sequal} \\
		& \sum_{j\in\mathcal{N}} \left(p_j - D_j\omega_j - d_j \fixes{-\frac{1}{R_j}\omega_j}\right) = 0 \label{eq:gfc_power} \\
		& \underline{d}_j\leq d_j \leq \overline{d}_j \quad \forall j\in\fixes{\mathcal{D}} \label{eq:gfc_bnds_dj}
	\end{align}
\end{subequations}
where \eqref{eq:gfc_sequal} is the excess computational power \eqref{eq:s_of_djs}, \eqref{eq:gfc_power} is the balance of electrical power on the grid, and \eqref{eq:gfc_bnds_dj} are the box constraints on the datacenter electrical power consumptions.


In order to take advantage of the structure of GFC \eqref{eq:gfc} we use the following mild assumptions.
\vspace{-0.05in}
\begin{assumption}\label{as:convexity}
	$g(s)$ is strictly convex and twice continuously differentiable. For all $j\in\mathcal{N}$: $c_j(d_j)$ is strictly convex and twice continuously differentiable for all $d_j\in[\underline{d}_j,\overline{d}_j]$.
\end{assumption}

\begin{assumption}\label{as:feas}
	GFC \eqref{eq:gfc} has a feasible solution, and for any optimal solution there exists a feasible $\boldsymbol{\theta}$ such that:
	\vspace{-0.05in}
	\begin{align}
		F_j(\boldsymbol{\theta}) = p_j-D_j\omega_j-d_j \fixes{-\frac{1}{R_j}\omega_j}, \quad \forall j\in\mathcal{N}.
	\end{align}
\end{assumption}

Convex cost functions are found in geographic load balancing optimization problems \cite{liu2011greening} and are consistent with concave disutility functions used in demand response programs \cite{bitar12,jiang11,li2011optimal}.
Assumption \ref{as:feas} ensures that for any optimal solution of power injections, there exists a set of voltage phase angles that satisfy the solution.


With the above assumptions we now show that GFC is a convex optimization problem.
\begin{lemma}\label{thm:gfc_convex}
	Given Assumption \ref{as:convexity}, then GFC \eqref{eq:gfc} is a convex optimization problem and has a unique solution.
\end{lemma}
\vspace{-0.1in}
\begin{proof}
	From Assumption \ref{as:convexity}, the objective function is strictly convex which means it has a unique minimizer.  Additionally, the equality constraints \eqref{eq:gfc_sequal} \eqref{eq:gfc_power} are linear and the inequality constraints \eqref{eq:gfc_bnds_dj} are convex which gives the result.
\end{proof}
\section{CHARACTERIZING THE OPTIMA}
\label{sec:optima}


We now provide characterizations of the optimal solution of GFC. This will then motivate the design of decentralized algorithm later on. 
From Assumption \ref{as:feas} and Lemma \ref{thm:gfc_convex}, the Karush-Kuhn-Tucker (KKT) conditions for optimality for GFC \eqref{eq:gfc} are applicable and can be determined with the dual variables $(\mu,\lambda,\underline{\boldsymbol{\kappa}},\overline{\boldsymbol{\kappa}})$ for each constraint respectively:
\begin{subequations}\label{eq:gfc_kkt}
	\begin{align}
		g'(s) - \mu = 0 & \label{eq:gfc_kkt_stat1} \\
		c_j'(d_j) + \mu a_j - \lambda - \underline{\kappa}_j + \overline{\kappa}_j = 0 & \quad \forall j\in\fixes{\mathcal{D}} \label{eq:gfc_kkt_stat2} \\
		\left(D_j \fixes{+\frac{1}{R_j}} \right)\omega_j-\lambda \left(D_j  \fixes{+\frac{1}{R_j}} \right) = 0 & \quad \forall j\in\mathcal{N} \label{eq:gfc_kkt_stat3} \\
		\underline{\kappa}_j(\underline{d}_j-d_j) = 0 & \quad \forall j\in\fixes{\mathcal{D}} \label{eq:gfc_kkt_cs1} \\
		\overline{\kappa}_j(d_j-\overline{d}_j) = 0 & \quad \forall j\in\fixes{\mathcal{D}} \label{eq:gfc_kkt_cs2} \\
		\underline{\kappa}_j \geq 0 & \quad \forall j\in\fixes{\mathcal{D}} \label{eq:gfc_kkt_dual1} \\
		\overline{\kappa}_j \geq 0 & \quad \forall j\in\fixes{\mathcal{D}} \label{eq:gfc_kkt_dual2} \\
		\sum_{j\in\fixes{\mathcal{D}}}a_jd_j - b - s = 0 \label{eq:gfc_kkt_prim1} \\
		\sum_{j\in\mathcal{N}} \left(p_j  - D_j\omega_j - d_j \fixes{-\frac{1}{R_j}\omega_j}\right) = 0 & \label{eq:gfc_kkt_prim2} \\
		\underline{d}_j-d_j \leq 0 & \quad \forall j\in\fixes{\mathcal{D}} \label{eq:gfc_kkt_prim3} \\
		d_j-\overline{d}_j \leq 0 & \quad \forall j\in\fixes{\mathcal{D}} \label{eq:gfc_kkt_prim4}
	\end{align}
\end{subequations}
where \eqref{eq:gfc_kkt_stat1} \eqref{eq:gfc_kkt_stat2} \eqref{eq:gfc_kkt_stat3} are the first-order stationary conditions, \eqref{eq:gfc_kkt_cs1} \eqref{eq:gfc_kkt_cs2} are the complementary slackness conditions, \eqref{eq:gfc_kkt_dual1} \eqref{eq:gfc_kkt_dual2} are the dual feasibility conditions, and \eqref{eq:gfc_kkt_prim1} \eqref{eq:gfc_kkt_prim2} \eqref{eq:gfc_kkt_prim3} \eqref{eq:gfc_kkt_prim3} are the primal feasibility conditions.
Note that the operator $(\cdot)'$ denotes the derivative.


From the above KKT conditions we can infer the following properties of the optimal solution:
\begin{enumerate}
	\item From \eqref{eq:gfc_kkt_stat3} we have that
	\vspace{-0.05in}
	\begin{align}
		\omega_j = \lambda & \quad \forall j\in\mathcal{N} \label{eq:lambda}
	\end{align}
	which means that the deviation of frequency at each bus is equal to a single value, the same as the equilibrium condition \eqref{eq:eqlbrm_eq_omega}.
	\revjournal{Note that $\lambda$ is the dual variable of \eqref{eq:gfc_power}, which means that the optimal frequency deviation equals the marginal cost for the network of datacenters to provide this steady-state load adjustment.}
	\item From \eqref{eq:gfc_kkt_stat1},
	\revjournal{we have that $\mu$ is equal to the marginal interdependent cost with respect to the excess computational power of the network:
		\vspace{-0.1in}
	\begin{align}
		\mu = g'(s)
	\end{align}
	Together with}
	Assumption \ref{as:convexity}, we have that at equilibrium the excess computational power $s$ is equal to:
	\vspace{-0.05in}
	\begin{align}
		s = (g')^{-1}(\mu) \label{eq:opt_prop2}
	\end{align}
	which is an increasing function of $\mu$ due to the strict convexity property.
	Note that the operator $(\cdot)^{-1}$ denotes the inverse.
	\item If $\underline{d}_j<d_j<\overline{d}_j$ then from \eqref{eq:gfc_kkt_cs1} \eqref{eq:gfc_kkt_cs2} we have that $\underline{\kappa}_j=\overline{\kappa}_j=0$.
	With \eqref{eq:gfc_kkt_stat2} \eqref{eq:gfc_kkt_stat3} we have that $\omega_j$ is equal to the marginal independent cost and its share of the marginal interdependent cost:
	\vspace{-0.1in}
	\begin{align}
		\omega_j= c'_j(d_j) + a_j\mu
	\end{align}
	The $a_j$ converts the marginal interdependent cost from being with respect to the excess computational power of the network into being with respect to load $j$.
	This indicates that even though all datacenters at equilibrium experience the same frequency deviation, they may not have the same marginal independent cost due to different burdens of the marginal interdependent cost.
	By applying \eqref{eq:lambda}, we have
	\vspace{-0.1in} 
	\begin{align}
		\lambda= c'_j(d_j) + a_j\mu
	\end{align}	
	Together with Assumption \ref{as:convexity}, we have that
	\vspace{-0.05in}
	\begin{align}
		d_j = (c'_j)^{-1}(\omega_j-a_j\mu) \label{eq:opt_prop3}
	\end{align}
	which is an increasing function of $(\omega_j-a_j\mu)$ due to the strict convexity property.
\end{enumerate}
\section{DISTRIBUTED FREQUENCY CONTROL}
\label{sec:algorithm}
In this section, we first state the distributed control laws and then prove their optimality and stability in solving GFC.

\subsection{Control Laws}


At time $t$, let $\omega_j(t)$ be the local frequency deviation measured by datacenter $j$, and let $s(t)$ be the excess computational power of the network measured by the cloud provider's workload balancer.
We propose the following control law of demand response $d_j(t)$ at each $j\in\mathcal{D}$, with an auxiliary variable $\mu (t)$ being broadcast by the cloud provider:
\vspace{-0.05in}
\begin{subequations}\label{eq:control}
	\begin{align}
		d_j(t) & = 	\left[(c'_j)^{-1}(\omega_j(t)-a_j\mu(t))\right]_{\underline{d}_j}^{\overline{d}_j} \quad \forall j\in\fixes{\mathcal{D}} \label{eq:control_dj} \\
		\mu(t) & = \mu(0) + \beta\int_{0}^{t}\left(s(\tau)-(g')^{-1}(\mu(\tau))\right)d\tau. & \label{eq:control_mu}
	\end{align}
\end{subequations}
The constant $\beta>0$ is a control parameter that determines the sensitivity that the auxiliary variable $\mu(t)$ is to the mismatch between the measured value of the excess computational power $s(t)$ and the value according to Equation \eqref{eq:opt_prop2}.
Note that $(g')^{-1}(\cdot)$ and $(c'_j)^{-1}(\cdot)$ are well defined because of the strict convexity assumption stated in Assumption \ref{as:convexity}.


\revjournal{The control laws were inspired by the Subgradient Method which was developed for solving 
convex optimization problems with non-differentiable objective functions (See \cite{bertsekas1999nonlinear} Chapter 6).}
Essentially, the control laws work by first
\revjournal{applying the measured frequency deviation $\omega_j(t)$ and the received auxiliary variable $\mu(t)$ to Equation \eqref{eq:opt_prop3} as if the system were already at equilibrium and implementing $d_j(t)$ as such.}
It then gradually
\revjournal{moves the auxiliary variable $\mu(t)$ in the direction of the difference between the LHS and RHS of Equation \eqref{eq:opt_prop2} to push the system towards equilibrium.
This is analogous to the Subgradient Method being applied in continuous time where the subgradient is measured by $(s(\tau)-(g')^{-1}(\mu(\tau))$ and $\beta$ is equivalent to the step size in discrete time.}

Note that by using $\mu(t)$ in \eqref{eq:control_dj} we assume that the communication delay is negligible, which is a mild assumption in many scenarios as the inter-datacenter delay is at the millisecond level \cite{moniz2017blotter}.
\fixes{However, we also evaluate the cases with significant delay in Section \ref{sec:perf_sensitivity}.}

In order to better qualify an equilibrium point, we further define it to include the auxiliary variable $\mu$ when the optimal property \eqref{eq:opt_prop2} is satisfied and therefore is not changing with time.
The following definition will be useful when proving stability in Section \ref{sec:stability}.
\begin{definition}\label{def:eqlbrm2}
		A closed-loop equilibrium of the system \eqref{eq:FreqDev} \eqref{eq:pj_def} \eqref{eq:NetFlow} \eqref{eq:s_of_djs} \eqref{eq:Swing_full2} \eqref{eq:control_mu}, is any solution $(\boldsymbol{\theta}^*,\boldsymbol{\omega}^*,\mathbf{P}^*,\mathbf{d}^*,s^*,\mu^*)$ that satisfies Definition \ref{def:eqlbrm} and
		\vspace{-0.05in}
		\begin{align}
			\frac{d\mu^*}{dt}=0 \label{eq:eqlbrm2_mu}.
		\end{align}
\end{definition}

\subsection{Optimality}


\revjournal{At equilibrium, Equation \eqref{eq:opt_prop2} becomes satisfied and so the auxiliary variable $\mu(t)$ remains constant according to control law \eqref{eq:control_mu}.
Also, Equation \eqref{eq:opt_prop3} was always satisfied because of control law \eqref{eq:control_dj}.}
The following theorem
\revjournal{further explains this point and}
states that an equilibrium point of the above distributed control laws is optimal to the GFC optimization problem \eqref{eq:gfc}.

\begin{theorem}\label{thm:optimality}
	Given Assumptions \ref{as:convexity} and \ref{as:feas}, an equilibrium point from Definition \ref{def:eqlbrm2} of the system \eqref{eq:FreqDev} \eqref{eq:pj_def} \eqref{eq:NetFlow} \eqref{eq:s_of_djs} \eqref{eq:Swing_full2} with control laws \eqref{eq:control} is an optimal solution of GFC \eqref{eq:gfc}.
\end{theorem}
\vspace{-0.05in}
\lversion{Please see Appendix \ref{sec:appendix_opt_proof} for the proof of the above theorem.}

\sversion{Please see the Technical Report~\cite{gcfc2019tech} for the proof of the above theorem.}

This is important because it shows that when the system reaches steady state, there is a guaranteed optimal balance between datacenter PFC load participation and the system wide deviated frequency that is equal to the deviated frequencies at every bus.
Also at steady state, the marginal independent cost for each datacenter is the equilibrium deviation of frequency discounted by its marginal contribution to the interdependent cost (See \eqref{eq:opt_prop2}, \eqref{eq:opt_prop3}).
This means that a datacenter with a large marginal contribution to the interdependent cost results in a low marginal independent cost as compared to the other datacenters.

\subsection{Stability}
\label{sec:stability}


To prove that the system is asymptotically stable with the distributed control laws, we give the following assumption which is found to be true under normal operating conditions \cite{bergen1981structure}.

\begin{assumption}\label{as:line_angles}
	The equilibrium phase angle deviations between connected buses are bounded: $|\theta^*_i-\theta^*_j|<\frac{\pi}{2}$ for all $(i,j)\in\mathcal{E}$.
\end{assumption}


We use the Lypanunov method in the following theorem to prove that an equilibrium point, within the neighborhood described in the above assumption, is asymptotically stable.
Additionally, we show that the system asymptotically converges to an equilibrium point that ignores the specific phase angles.
This is important because it guarantees that the system trajectory is always moving towards an equilibrium point.

\begin{theorem}\label{thm:stability}
	Given Assumptions \ref{as:convexity} and \ref{as:feas}, an equilibrium point (Definition \ref{def:eqlbrm2}) of the system \eqref{eq:FreqDev} \eqref{eq:pj_def} \eqref{eq:NetFlow} \eqref{eq:s_of_djs} \eqref{eq:Swing_full2} with feedback control \eqref{eq:control} is \fixes{locally asymptotically stable within the neighborhood around $\boldsymbol{\theta}^*$ if the equilibrium satisfies Assumption \ref{as:line_angles}.}
	In particular under the same assumptions, the trajectory of $(\boldsymbol{\omega},\mathbf{P},\mathbf{d},s,\mu)$ such that $|\theta_i-\theta_j|<\frac{\pi}{2}:\forall(i,j)\in\mathcal{E}$ will asymptotically converge to an equilibrium point $(\boldsymbol{\omega}^*,\mathbf{P}^*,\mathbf{d}^*,s^*,\mu^*)$.
\end{theorem}

\lversion{Please see Appendix \ref{sec:appendix_stability_proof} for the proof of the above theorem.}

\sversion{Please see the Technical Report~\cite{gcfc2019tech} for the proof of the above theorem.}
\section{PERFORMANCE EVALUATION}
\label{sec:perf}

\fixes{We simulate a 400 MW generation loss at the 5 second time stamp and show that the proposed feedback control with existing droop control stabilizes the system to an equilibrium point closer to the nominal frequency than droop control alone.
We compare the proposed control with a decentralized load control OLC~\cite{zhao2014optimal} and show that our proposed control has significant participation cost savings for the network of data centers.}

\subsection{Setup}
In order to demonstrate the performance of the proposed control on a more realistic system model than described in Section \ref{sec:model}, we use the Power System Toolbox (PST)~\cite{chow2009power} to simulate it.

\begin{figure}
	\centering
	\includegraphics[width=0.9\linewidth]{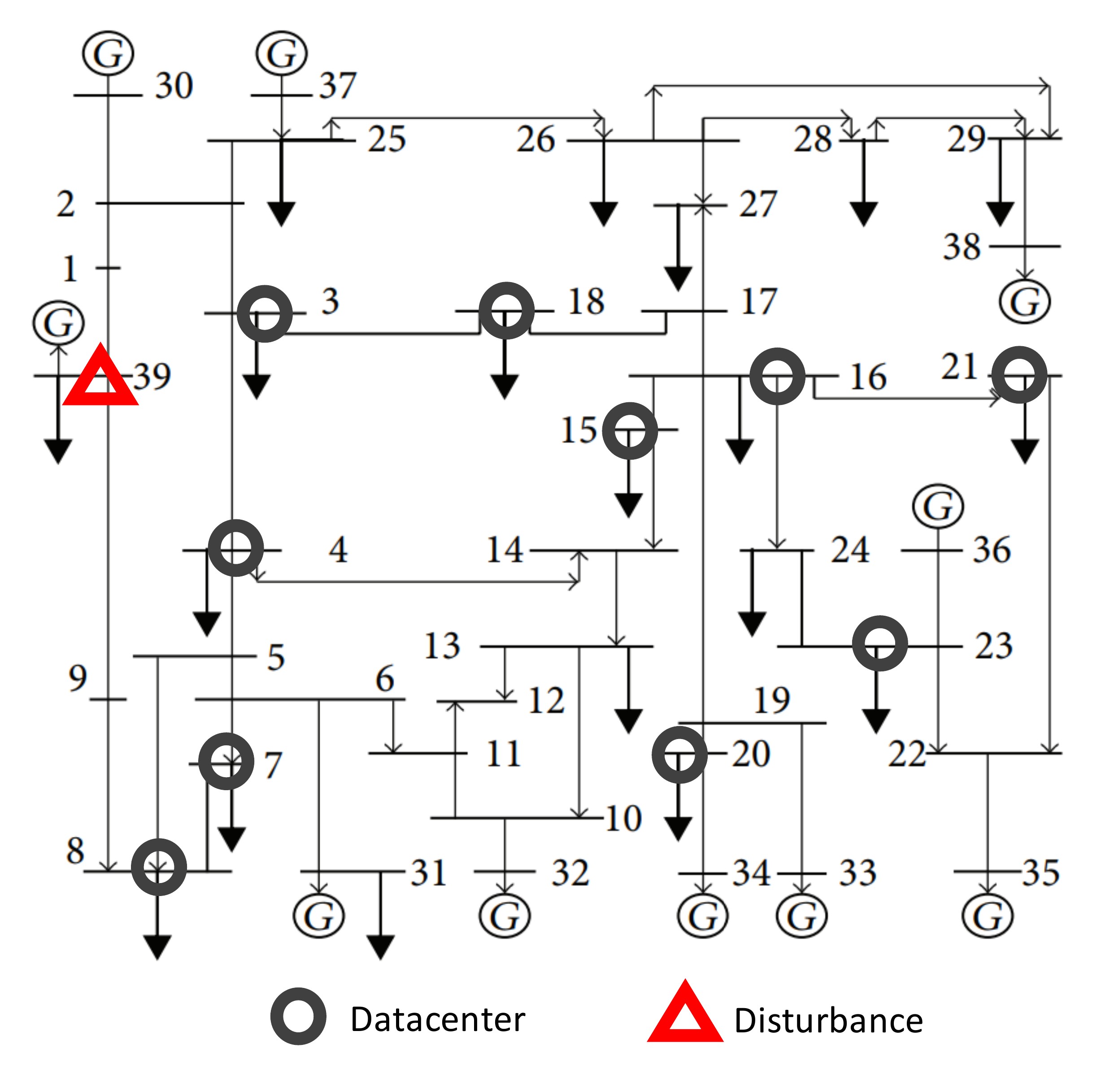}
	\vspace{-0.1in}
	\caption{The network of 10 datacenters on IEEE 39-bus test system \cite{zhang2015Optimized}.
	The total power demand is 14 GW, and the total datacenter power demand is 250 MW.}
	\label{fig:39bus}
\end{figure}

\paragraph*{Power network}
The IEEE 39-bus (New England) system was chosen as the test case for the evaluation which has 19 buses available to place controllable loads.
The total power demand is set at 14 GW~\cite{isone2016energy}.
We use the following ten buses to place controllable datacenter loads: 3, 4, 7, 8, 15, 16, 18, 20, 21, and 23.
\fixes{Additionally, the generator buses 30-39 all have droop control settings of $R_j=0.05$ which is a typical setting~\cite{undrill2018primary}.}
To simulate a power disturbance, at the 5 second time stamp power drops \fixes{400 MW} from a generator at bus 39.
\revjournal{Figure \ref{fig:39bus} depicts the power network setup}.

\paragraph*{Network of datacenters}
The datacenters each has a nominal demand of 25 MW which in total is 1.8\% of the total power demand~\cite{shehabi2016united}.  Each datacenter has minimum and maximum demands of 15 MW and 30 MW respectively.
Each datacenter's efficiency $a_j$ is estimated based on the fact that typical datacenters have a Power Usage Effectiveness (PUE) range between 1.1 and 2.1 with an average of 1.8~\cite{cheung2014energy,shehabi2016united}, therefore we randomly select $1/a_j\in[1.1,2.1]$ with $\mathbb{E}[1/a_j]=1.8$.

\paragraph*{Costs}
The objective function used in analyzing the cost for control is defined as follows:
\vspace{-0.05in}
\begin{align}
	\gamma\frac{1}{2}\left(\fixes{\left(-\sum_{j\in\mathcal{D}}a_j\delta_j\right)^+}\right)^2 + \sum_{j\in\mathcal{D}}\frac{\eta_j}{2}\delta_j^2 + \alpha\sum_{j\in\mathcal{N}}\frac{D_j\fixes{+\frac{1}{R_j}}}{2}\omega_j^2. \label{eq:obj_sim}
\end{align}
where \fixes{$(x)^+:=\max\{0,x\}$ and} $\delta_j:=d_j-25$ which means that each datacenter has a nominal demand of 25 MW.
The first term is the interdependent cost, the second term is the independent costs, and the third term is the cost associated with the deviations of frequency.
The interdependent cost is approximated by the fact that the Amazon Web Service's revenue of 1.3 million servers \cite{aws_capacity} is \$2.5 billion for the first quarter in 2016 \cite{aws_revenue}.
This gives a cost of \$2.47 per second for each 10k servers.
For our system, this translates into maximum of estimated interdependent cost for each second is $\$123.50$. Hence assuming an average PUE of 1.8 and only half of a 100 MW power equivalent workload is processed, we set $\gamma=\$0.16/\text{MW}^2$ throughout the evaluation.
The dependent cost for each datacenter includes its wasted cost for under/over-utilizing the pre-purchased electricity, operation, and maintenance.
Based on the Total Cost of Ownership in \cite{barroso2013datacenter}, under utilizing a datacenter by 50\% can cost $\$0.28$ per second for each 10k servers.  For the average datacenter in our system with 50k servers this is $\$1.40$ which we equate to a 5 MW decrease of the total available 10 MW decrease.
Hence, we randomly choose $\eta_j$ such that $\mathbb{E}[\eta_j]=\$0.11/\text{MW}^2$.
The cost for deviated frequency has been valued at \$15/MW for a 0.2 Hz deviation~\cite{nationalgrild2016enhanced,lin2016energy}.
Since $\sum_{j\in\mathcal{N}}\frac{D_j\fixes{+\frac{1}{R_j}}}{2}\omega_j$ is the aggregate frequency-sensitive load \fixes{with droop control}, we set $\alpha$ to \$75/MW-Hz.

\paragraph*{Baselines}
To show the benefit of incorporating interdependent load costs, we compare the proposed control to Optimal Load Control (OLC) described in \cite{zhao2014optimal} which is a decentralized control that does not take into account interdependent costs.
\fixes{In the stability analysis, we also compare it to using droop control only without load participation.}
In the equilibrium cost analysis, we also compare it to the case called ``lower bound" which is an estimated offline optimal solution to GFC \eqref{eq:gfc} with $D_j=0:\forall j\in\mathcal{N}$.

\begin{figure*}[!ht]
	\begin{center}
		\subfigure[Frequency]{{\includegraphics[width=0.66\columnwidth]{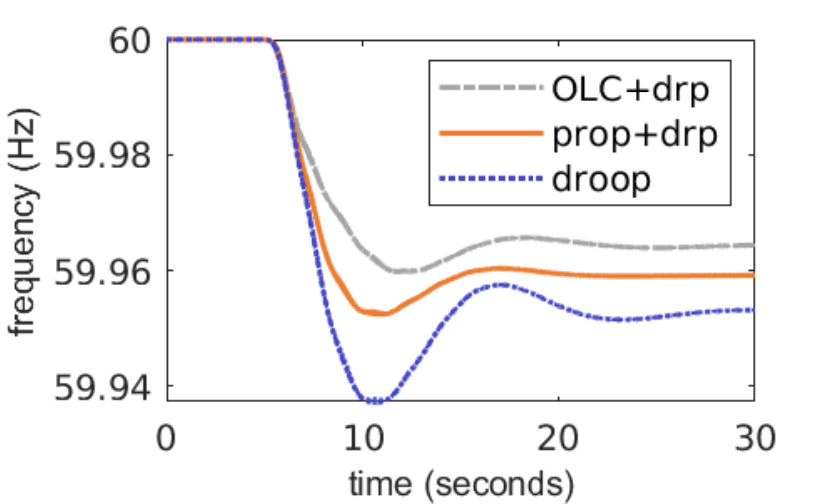}}
			\label{fig:load_freq}} \hspace{-0.1in}		
		\subfigure[Power demand change, $\delta_j$]{{\includegraphics[width=0.6\columnwidth]{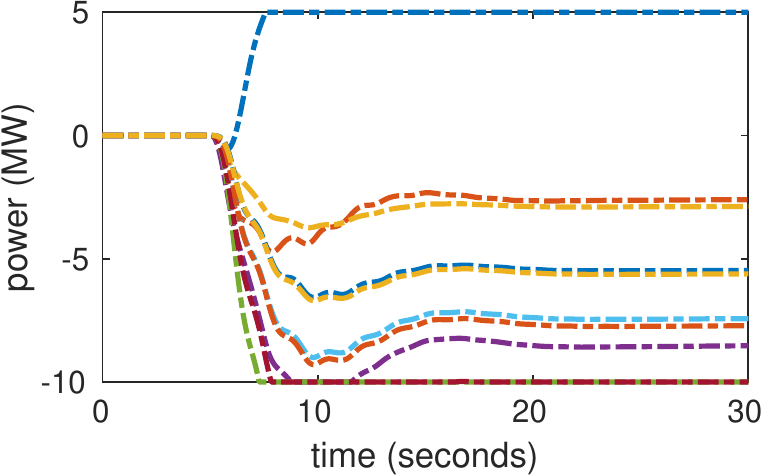}}
			\label{fig:load}}
		\subfigure[Cost change]{{\includegraphics[width=0.6\columnwidth]{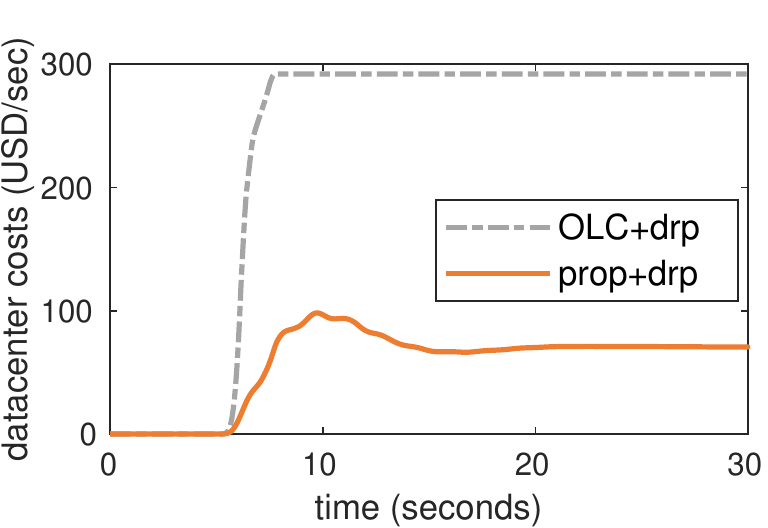}}
			\label{fig:cost_time}}
		\vspace{-0.05in}					
		\fixes{\caption{Trajectories of the system state variables: (a) Bus frequencies for each of the ten buses containing a datacenter under the three different control schemes; (b) Changes in load for each of the ten datacenters under the proposed control; (c) Total cost changes for the network of datacenters under OLC and the proposed control.}}
	\end{center}
	\vspace{-0.05in}
\end{figure*}

\begin{figure*}[!ht]
	\begin{center}
		\subfigure[Total cost]{{\includegraphics[width=0.63\columnwidth]{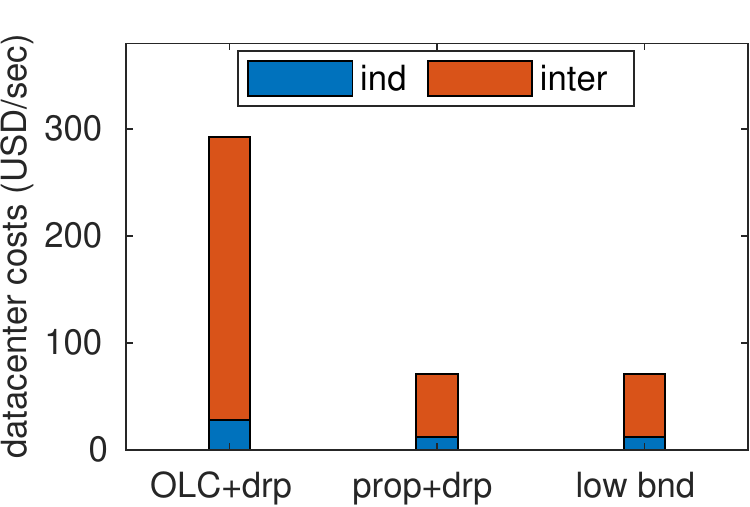}}
			\label{fig:breakdown_cost}}					
		\subfigure[Independent costs]{{\includegraphics[width=0.62\columnwidth]{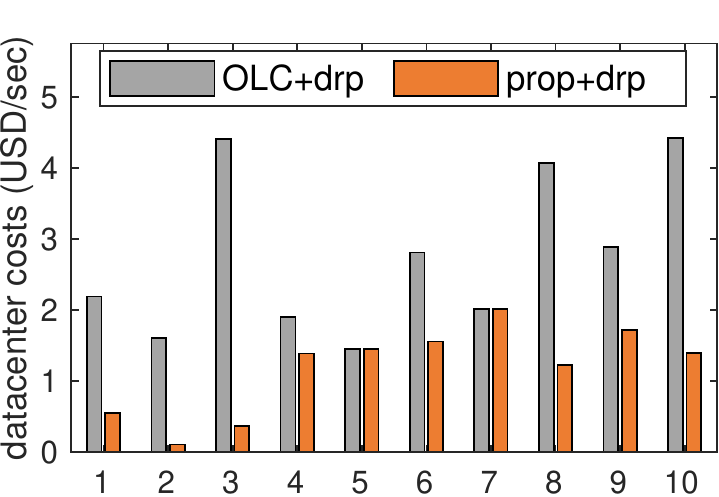}}
			\label{fig:cost}}
		\subfigure[Load deviations]{{\includegraphics[width=0.65\columnwidth]{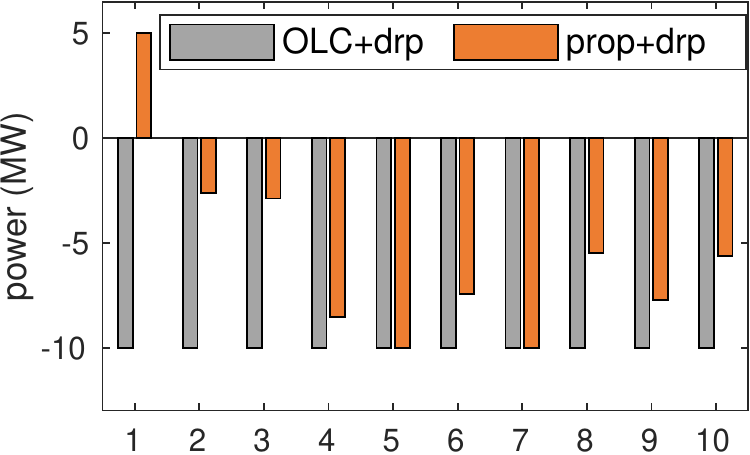}}
			\label{fig:load_deviation}}
		\vspace{-0.05in}	
		\fixes{\caption{(a) Total equilibrium cost separated by independent and interdependent costs of OLC, the proposed control, and the lower bound. For each datacenter (b) and (c) give the independent costs and equilibrium load deviations respectively under OLC and the proposed control.  Note: The datacenters are numbered in decreasing order of efficiency $a_j$.}}
	\end{center}
	\vspace{-0.05in}
\end{figure*}

\begin{figure*}[!ht]
	\begin{center}
		\subfigure[Impact of communication delay]{{\includegraphics[width=0.65\columnwidth]{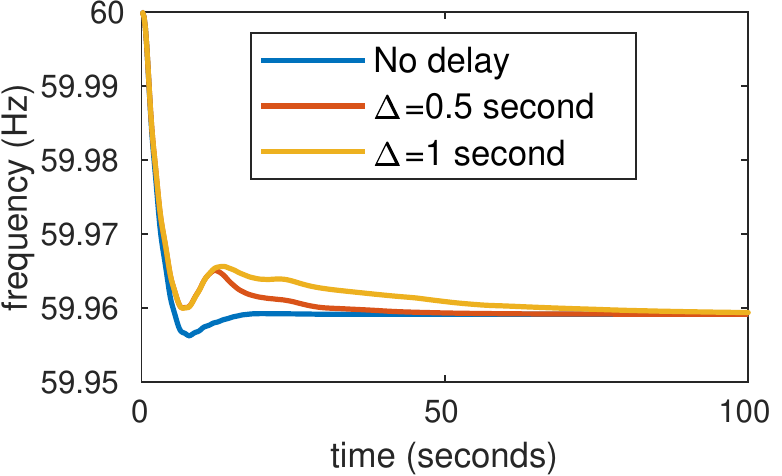}}
			\label{fig:delay_freq}}
		\subfigure[Equilibrium Frequency Deviation ]{{\includegraphics[width=0.65\columnwidth]{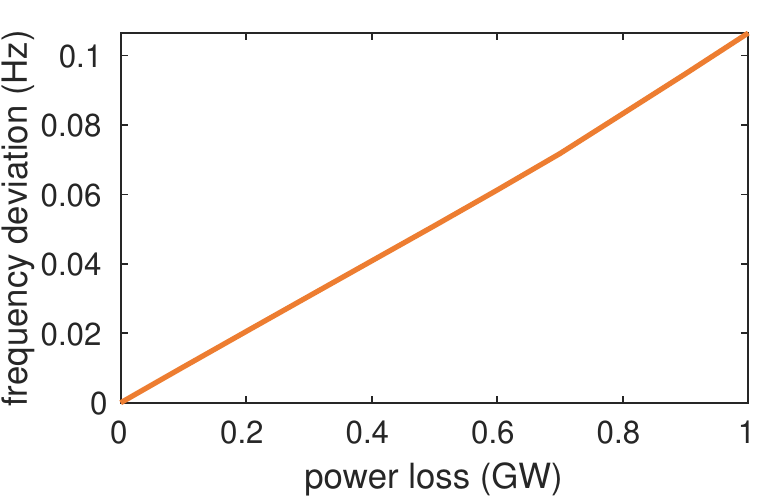}}
			\label{fig:power_freq_dev}}
		\subfigure[Total Cost]{{\includegraphics[width=0.65\columnwidth]{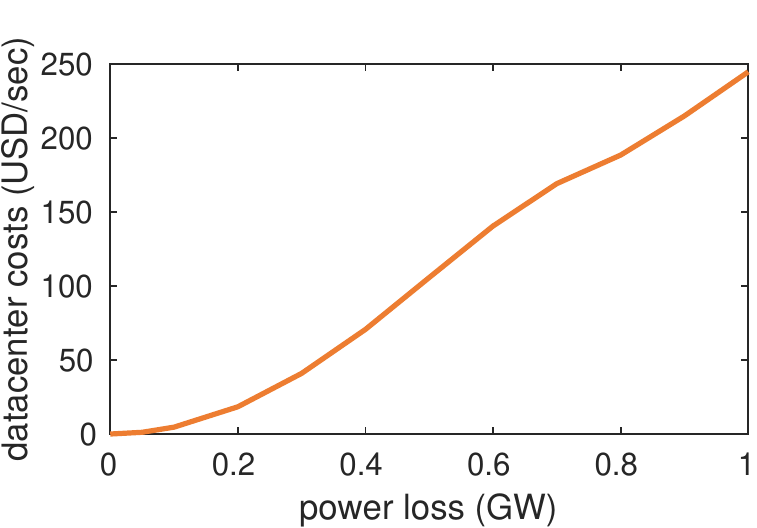}}
			\label{fig:power_costs}}
		\vspace{-0.05in}							
		\fixes{\caption{(a) Stability analysis with the presence of communication delay $\Delta$ and setting control actions every 0.1 seconds. Sensitivity analysis on disturbance size with respect to (a) Equilibrium frequency deviation and (b) Cost to the network of datacenters.}}
		\label{fig:sensitivity_analysis2}
	\end{center}
	\vspace{-0.05in}
\end{figure*}

\lversion{
\begin{figure*}[!ht]
	\begin{center}
		\subfigure[Impact of $\gamma$ (interdependent cost)]{{\includegraphics[width=0.59\columnwidth]{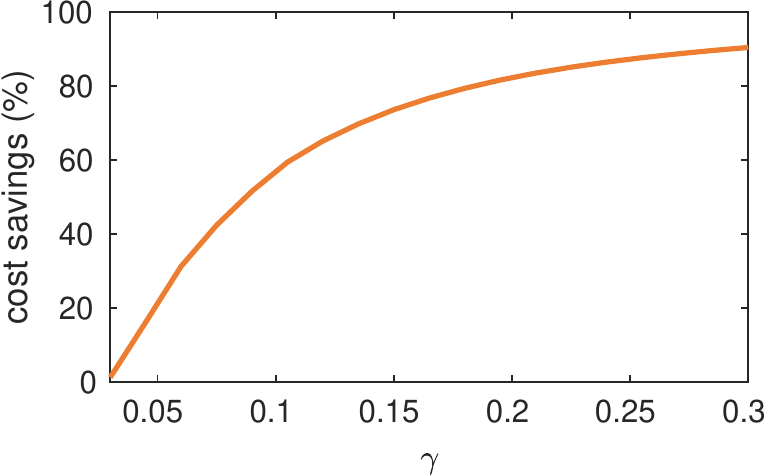}}
			\label{fig:gamma_costs}}
		\subfigure[Impact of demand flexibility]{{\includegraphics[width=0.59\columnwidth]{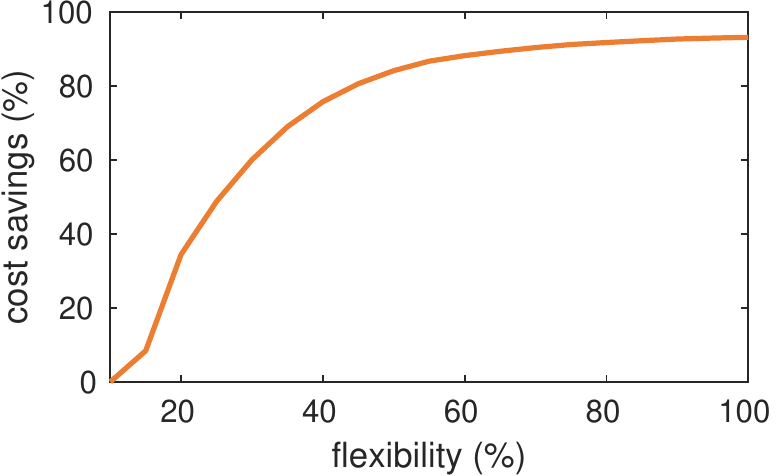}}
			\label{fig:flex_costs}}
		\subfigure[Frequency Deviation \& Total Cost]{{\includegraphics[width=0.64\columnwidth]{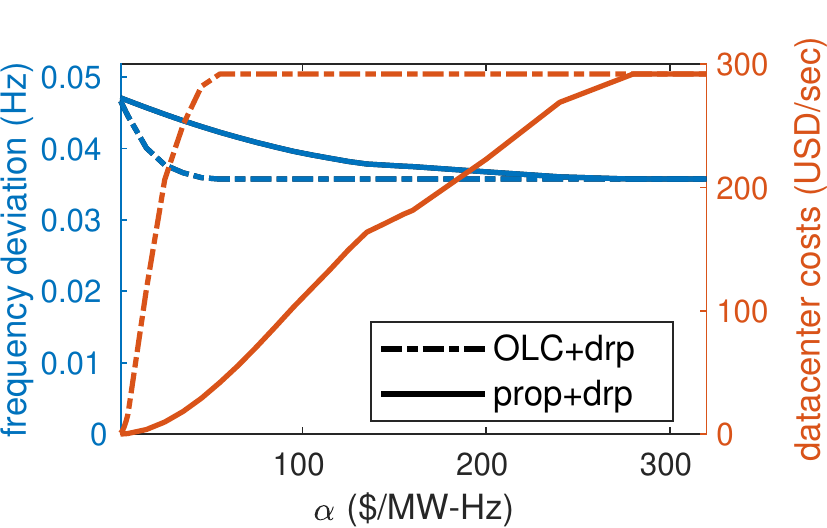}}
			\label{fig:weight_tradoff}}
		\vspace{-0.05in}							
		\fixes{\caption{Sensitivity analysis in terms of cost savings on the network of datacenters by using the proposed control instead of OLC for: (a) Interdependent cost coefficient; (b) Demand flexibility.
				(c) Trade-off between the equilibrium frequency deviation and cost to the network of datacenters by varying the frequency deviation cost coefficient $\alpha$.}}
		\label{fig:sensitivity_analysis}
	\end{center}
	\vspace{-0.05in}
\end{figure*}
}

\subsection{Stability analysis}


\vspace{-0.05in}
\fixes{The proposed control stabilizes the frequency and each load decision within 20 seconds of the generation loss whereas droop control alone is not fully stable until after 25 seconds have passed (Figure \ref{fig:load_freq}).
Also, the additional load participation from the network of datacenters helps decrease the deviation in frequency as compared to droop control only, which is widely used now.
Since both the proposed control and OLC converge to an equilibrium frequency within 20 seconds, this shows that incorporating interdependent costs does not negatively impact the speed at which load participation helps stabilize the power network.
Note that for each control scheme in Figure \ref{fig:load_freq}, the frequencies for the separate buses are so close that they are hard to be distinguished from one another.}


Additionally in Figure \ref{fig:load}, all of the loads reach their equilibrium points at different times which are all within \fixes{20 seconds} of the generation loss.
It is worth noting that the most efficient datacenter DC 1 actually increases its load instead of decreasing it.
This is because the value of $(\omega_1-a_1\mu)$ remains positive throughout \fixes{most of} the simulation due to the product of a negative auxiliary variable $\mu$ and a relatively large computational efficiency $a_1$ counteracting and surpassing the negative deviation of frequency $\omega_1$.
That positive value is used in the proposed control \eqref{eq:control_dj} which is an increasing function.

As datacenters participate in PFC, they suffer the costs from decreasing their power consumption.
Figure \ref{fig:cost_time} shows that \fixes{their total cost under} both OLC and the proposed control increase right after the incident happens.
The cost of our proposed control is
less than \fixes{25\%} compared to OLC. 
This is because our method takes the interdependent cost into account while OLC does not.



\subsection{Equilibrium cost analysis}


\fixes{The proposed control shows a 75\% decrease in total cost incurred by the network of datacenters from OLC and meets the lower bound since the control laws are solving GFC \eqref{eq:gfc} at the equilibrium point (See Figure \ref{fig:breakdown_cost}).}
While on the other hand, it has a larger deviation of frequency than OLC (See Figure \ref{fig:load_freq}) because OLC underestimates the actual total cost of decreasing the load consumptions.
This is caused by OLC minimizing only for independent and deviated frequency costs (i.e. last two terms of \eqref{eq:obj_sim}) as compared to the proposed control which minimizes the sum of all three cost types.
\fixes{This results in OLC having a 15\% smaller equilibrium deviation of frequency than the proposed control.}

Also, the proposed control will utilize smaller power deviations from higher efficient datacenters than OLC.
\fixes{In fact, Figure \ref{fig:cost} shows that OLC uses all of the available capacity in the datacenters whereas the proposed control implements cost-balanced deviations.}
Higher values of computational efficiency $a_j$ cause power deviations to have larger marginal increases to the interdependent cost (i.e. the first term of \eqref{eq:obj_sim}) than for lower values of $a_j$.
\fixes{This causes the datacenters with higher than average efficiencies (DC 1-3) to have less negative deviations in power (See Figure \ref{fig:load_deviation}).}



\subsection{Sensitivity analysis}
\label{sec:perf_sensitivity}


\fixes{\paragraph*{Communication delay ($\Delta$ seconds)}
When broadcasting out the value of $\mu$ through the communication network, it is possible for there to be some short delay.
Since delays may cause slow stabilization of the system, we evaluate their impacts on the frequency's convergence to equilibrium.
To do this, we implement the same proposed control laws as before but instead of using the instantaneous $\mu$, we use $\mu$ that is $\Delta$ seconds old.
The delays tested are 0.5 and 1.0 seconds and are compared against the case with no delay.
Additionally, we also take into account how often each datacenter can change its power decision by discretizing the control actions into 0.1 second timeslots which is reasonable for controlling servers~\cite{meisner2009powernap}.
Figure \ref{fig:delay_freq} shows that proposed control converges well even in the presence of communication delays.
It is interesting to note that although the trajectories with the delay converge slower at the tail end of the time, they don't dip down as low as the trajectory without delay.
This is because with a communication delay, the datacenters react immediately without knowing the impact on the interdependent cost due to $\mu$ be $\Delta$ seconds out of date.
But as $\mu$ gets updated in time, the datacenters gradually take the interdependent costs into account.}

\fixes{\paragraph*{Disturbance Size}
Since power networks can experience different disturbance sizes, we vary the size of the power loss at bus 39 between 100 MW and 1 GW.
As expected, a larger power loss results in a larger in equilibrium deviation in frequency (See Figure \ref{fig:power_freq_dev}) that in turn causes larger costs to the network of datacenters due to the control laws (See Figure \ref{fig:power_costs}).
Additionally, since both the equilibrium deviation and the total cost increase approximately linearly with respect to the power loss, this shows that the tradeoff between equilibrium frequency and the total cost to the datacenters is also approximately linear.
Note that during normal operation of the power network, the controllable loads would constantly change by small values.
To counteract degradation of the equipment's lifetime, data centers can optimize both energy efficiency and lifetime via DVFS~\cite{das2014combined}.}

\lversion{\paragraph*{Impact of interdependence}
In order to measure the effect that the interdependent cost has on the cost savings of the proposed control, we vary $\gamma\in[0.03,0.3]$ (See Figure \ref{fig:gamma_costs}).
While low interdependent costs ($\gamma\sim 0.03$) make the cost savings between the proposed control and OLC insignificant, higher interdependent costs result in larger cost savings for the proposed control.}

\lversion{\paragraph*{Demand flexiblity ($X$\%)} 
As datacenters may not be able to reduce all of their demand at this fast timescale,
next we evaluate the impacts of datacenters' demand flexibility, which is the fraction of the nominal load that can be changed, i.e. $d_j \in [25*(1-X/100),30]$ (cf. Figure \ref{fig:flex_costs}).
Observe that the cost savings remain significant for most of the range except just \fixes{below 20\%} demand flexibility.
This is where both proposed and OLC total costs converge because the box constraints \eqref{eq:gfc_bnds_dj} become active and every datacenter reduces its demand to its lowest allowed value.}

\lversion{\fixes{\paragraph*{Frequency deviation cost coefficient $\alpha$}
The extent to which the network of datacenters participates in PFC and incurs the cost of doing so depends directly on the cost of the frequency deviation to the power network as expressed by the cost coefficient $\alpha$.
By varying $\alpha$ between 0 and 300 \$/MW-Hz we can see the trade-off between the equilibrium frequency deviation and the cost incurred by the network of datacenters (See Figure \ref{fig:weight_tradoff}).
At the extreme values of $\alpha$, the proposed control and OLC give the same trade-off that correspond to either no participation in PFC or expending all available load capacity.
However, between these extremes the proposed control is more conservative in its participation due to taking account of the interdependent costs, thus it incurs a lower cost with a higher equilibrium frequency deviation.}}


\section{CONCLUSION}
Frequency control is an important class of mechanisms to ensure high-quality power for the grid and has been an increasingly attractive option for load participation.
Prior work on distributed load control assumed that the cost for changing load demand at each geographic location is independent from the rest. 
However, in some networked systems such as a network of datacenters that support cloud computing, the decisions made at each datacenter affect the group because of interdependent costs and so require geographic coordination.
In this paper, we designed a set of distributed control laws
\revjournal{inspired by the subgradient method}
that can handle interdependent costs in such a way that is provably stable.
Also, we proved that the final equilibrium point optimally balances the cost of load participation with the frequency level deviated from its nominal set point.
We tested our control laws for a network of datacenters on a realistic emulator, Power System Toolbox. We found that there is significant cost savings with the proposed control law over existing benchmarks that do not account for interdependent costs, and the proposed control is robust to communication delays. 

\lversion{The results presented in this paper open up
\revjournal{four}
distinct future research directions.
The first is to explore how other interdependent systems (e.g. electric mass transit, thermal grids) can be used to help increase the reliability of the grid.
The second is to investigate how a network of datacenters that are located in multiple disjoint power grids can utilize their interconnectedness to enhance the reliability in those grids.
\revjournal{The third is to separate the computational workload into different resource demands, each with a different interdependent cost and computational efficiency.}
The fourth is to apply the distributed control laws to a system with a higher-order transient stability model.
For our future work, we plan to extend the proposed control laws to take into account further network effects such as power flow constraints across lines and network losses.}

\bibliographystyle{IEEEtran}
\bibliography{reference}


\appendix
\subsection{Proof of Theorem \ref{thm:optimality}}
\label{sec:appendix_opt_proof}
\begin{proof}
	Since GFC is a convex optimization problem (Lemma \ref{thm:gfc_convex}) and is feasible (Assumption \ref{as:feas}), we need to show that the proposition satisfies the KKT conditions for optimality \eqref{eq:gfc_kkt}.
	
	From the strict convexity property in Assumption \ref{as:convexity}, the function $(c'_j)^{-1}(\cdot)$ is well defined and increasing. Therefore, control law \eqref{eq:control_dj} can be separated into three cases: (i) $d_j^* \in (\underline{d}_j,\overline{d}_j)$, thus $d_j^* = (c'_j)^{-1}(\omega_j^*-a_j\mu^*)$; (ii) $d_j^*=\underline{d}_j$, thus $(c'_j)^{-1}(\omega_j^*-a_j\mu^*)\leq \underline{d}_j$; (iii) $d_j^*=\overline{d}_j$, thus $(c'_j)^{-1}(\omega_j^*-a_j\mu^*)\geq \overline{d}_j$.
	Case (i) results in:
	\begin{align}
		\underline{d}_j < (c'_j)^{-1}(\omega_j^*-a_j\mu^*) < \overline{d}_j  \nonumber \\
		c'_j(\underline{d}_j) < \omega_j^*-a_j\mu^* < c'_j(\overline{d}_j ). \nonumber
	\end{align}
	
	\noindent since $c'_j(\cdot)$ and $(c'_j)^{-1}(\cdot)$ are increasing functions.
	Likewise, Case (ii) results in:
	\begin{align}
		(c'_j)^{-1}(\omega_j^*-a_j\mu^*) & \leq \underline{d}_j \nonumber\\
		\omega_j^*-a_j\mu^* & \leq c'_j(\underline{d}_j) \nonumber
	\end{align}
	
	\noindent and Case (iii) results in:
	\begin{align}
		\overline{d}_j \leq (c'_j)^{-1}(\omega_j^*-a_j\mu^*) \nonumber\\
		c'_j(\overline{d}_j) \leq \omega_j^*-a_j\mu^*. \nonumber
	\end{align}
	
	\noindent Let us define the following two variables:
	\begin{align}
		\underline{\kappa}_j & :=[c'_j(\underline{d}_j)-(\omega_j^*-a_j\mu^*)]^+ \nonumber\\
		\overline{\kappa}_j & :=[(\omega_j^*-a_j\mu^*)-c'_j(\overline{d}_j)]^+ \nonumber
	\end{align}
	
	\noindent which satisfy \eqref{eq:gfc_kkt_dual1} \eqref{eq:gfc_kkt_dual2}.
	Note that in: Case (i), $\underline{d}_j-d_j^*\neq 0$, $d_j^*-\overline{d}_j\neq 0$, $\underline{\kappa}_j=0$, $\overline{\kappa}_j=0$; Case (ii), $\underline{d}_j-d_j^*=0$, $d_j^*-\overline{d}_j\neq 0$, $\underline{\kappa}_j\geq 0$, $\overline{\kappa}_j=0$; Case (iii) $\underline{d}_j-d_j^*\neq 0$, $d_j^*-\overline{d}_j= 0$, $\underline{\kappa}_j=0$, $\overline{\kappa}_j\geq 0$. Therefore each case satisfies \eqref{eq:gfc_kkt_cs1} \eqref{eq:gfc_kkt_cs2}.
	
	Additionally, each case from their definitions satisfies:
	\begin{align}
		c_j'(d_j^*) - \underline{\kappa}_j + \overline{\kappa}_j = \omega_j^* - a_j\mu^* \label{eq:gfc_kkt_proof_stat2}
	\end{align}
	
	\noindent and we can use the equilibrium condition \eqref{eq:eqlbrm_eq_omega} to define the following variable since each frequency deviation must be equal to a single value:
	\begin{align}
		\lambda:=\omega^*=\omega_j^* \quad \forall j\in\mathcal{N}. \nonumber
	\end{align}
	
	\noindent which is equivalent to \eqref{eq:gfc_kkt_stat3} when multiplied by $\left(D_j\fixes{+\frac{1}{R_j}}\right)$.
	Therefore, substituting $\lambda$ for $\omega_j^*$ in \eqref{eq:gfc_kkt_proof_stat2} becomes \eqref{eq:gfc_kkt_stat2}.
	
	
	The time derivative of \eqref{eq:control_mu} at equilibrium gives:
	\begin{align}
		\frac{d\mu^*}{dt} & = \beta\left(s^*-(g')^{-1}(\mu^*)\right) \nonumber
	\end{align}
	
	\noindent From the strict convexity property in Assumption \ref{as:convexity}, the function $(g')^{-1}(\cdot)$ is well defined. Since $d\mu^*/dt=0$ from equilibrium condition \eqref{eq:eqlbrm2_mu}, then we have that $(g')^{-1}(\mu^*)=s$ which is equivalent to \eqref{eq:gfc_kkt_stat1}.
	
	System equation \eqref{eq:s_of_djs} is equivalent to \eqref{eq:gfc_kkt_prim1}.
	
	To get \eqref{eq:gfc_kkt_prim2}, start with \eqref{eq:Swing_full2} and apply \eqref{eq:pj_def}, and equilibrium condition \eqref{eq:eqlbrm_omega_constant}:
	\begin{align}
		p_j - d_j^* - D_j\omega_j^* \fixes{-\frac{1}{R_j}\omega_j^*} - F_j(\boldsymbol{\theta}^*) = 0. \nonumber
	\end{align}
	
	\noindent Summing the above equation for all $j\in\mathcal{N}$ gets \eqref{eq:gfc_kkt_prim2} because $\sum_{j\in\mathcal{N}}F_j(\boldsymbol{\theta}^*) = 0$ from each term canceling when summing \eqref{eq:NetFlow}.
	
	Because the range of \eqref{eq:control_dj} is $[\underline{\delta}_j,\overline{\delta}_j]$, \eqref{eq:gfc_kkt_prim3} \eqref{eq:gfc_kkt_prim4} are satisfied.
	
	Since all of the KKT conditions have been satisfied, we get the resultant.
\end{proof}

\subsection{Lemmas}
The following Lemmas are used in the proof of Theorem \ref{thm:stability}.

\begin{lemma}\label{thm:dV1dt}
	The time derivative of $\mathcal{V}_1$ defined in \eqref{eq:stab_V1} with feedback control \eqref{eq:control} is \eqref{eq:stab_dV1dt}.
\end{lemma}

\begin{proof}
	The time derivative of \eqref{eq:control_mu} is:
	\begin{subequations}
		\begin{align}
			\frac{d\mu}{dt} & =  \beta\left(s-(g')^{-1}(\mu)\right) \\
			& = \beta\left(\sum_{j\in\mathcal{N}}a_jd_j - b-(g')^{-1}(\mu)\right) \label{eq:dmudt_proof}
		\end{align}
	\end{subequations}
	where the second equality comes from replacing $s$ with the RHS of \eqref{eq:s_of_djs} and setting that $d_j=\underline{d}_j=\overline{d}_j=0:\forall j\notin\mathcal{D}$.
	At equilibrium, the condition \eqref{eq:eqlbrm2_mu} can be applied to the above equation to give:
	\begin{align}
		\sum_{j\in\mathcal{N}}a_jd_j^* - b = (g')^{-1}(\mu^*) \label{eq:eqlbrm2_mu_proof}
	\end{align}
	
	Take the time derivative of \eqref{eq:stab_V1}:
	\begin{subequations}
		\begin{align}
			\frac{d\mathcal{V}_1}{dt} & = \frac{1}{\beta}(\mu-\mu^*)\frac{d\mu}{dt} \\
			& = (\mu-\mu^*) \left(\sum_{j\in\mathcal{N}}a_jd_j - b-(g')^{-1}(\mu)\right) \\
			& = (\mu-\mu^*) \left(\sum_{j\in\mathcal{N}}a_jd_j-\sum_{j\in\mathcal{N}}a_jd^*_j\right) \nonumber \\
			& \quad - (\mu-\mu^*)\left((g')^{-1}(\mu)-(g')^{-1}(\mu^*)\right) \nonumber\\
			& \quad +(\mu-\mu^*)\left(\sum_{j\in\mathcal{N}}a_jd^*_j-b-(g')^{-1}(\mu^*)\right) \\
			& = (\mu-\mu^*)\sum_{j\in\mathcal{N}}a_j(d_j-d^*_j) \nonumber \\
			& \quad - (\mu-\mu^*)\left((g')^{-1}(\mu)-(g')^{-1}(\mu^*)\right) \\
			& = \sum_{j\in\mathcal{N}}(a_j\mu-a_j\mu^*)(d_j-d^*_j) \nonumber \\
			& \quad - (\mu-\mu^*)\left((g')^{-1}(\mu)-(g')^{-1}(\mu^*)\right)
		\end{align}
	\end{subequations}
	The second equality comes from applying the differential equation \eqref{eq:dmudt_proof}.
	The third equality comes from adding $0=((g')^{-1}(\mu^*)-(g')^{-1}(\mu^*))$ and $0=\left(\sum_{j\in\mathcal{N}}a_jd^*_j-\sum_{j\in\mathcal{N}}a_jd^*_j\right)$ to the second factor and distributing out $(\mu-\mu^*)$.
	The fourth equality comes from combining the summations in the first term and from the equilibrium equation \eqref{eq:eqlbrm2_mu_proof} applied to the third term.
	The last equality comes from bringing $(\mu-\mu^*)$ into the summation and distributing $a_j$.
	Then after defining $d_j$ as the function from \eqref{eq:control_dj} we get \eqref{eq:stab_dV1dt}.
\end{proof}


The following Lemma is used in the proof of Theorem \ref{thm:stability} and comes from \cite{zhao2014optimal}(Theorem 2).

\begin{lemma}\label{thm:dV23dt}
	The time derivative of $\mathcal{V}_2+\mathcal{V}_3$ defined in \eqref{eq:stab_V2} \eqref{eq:stab_V3} for system \eqref{eq:FreqDev} \eqref{eq:pj_def} \eqref{eq:NetFlow} \eqref{eq:Swing_full2} with feedback control \eqref{eq:control_dj}, is \eqref{eq:stab_dV23dt}.
\end{lemma}

\oproof{
\begin{proof}
	Take the time derivative of $\mathcal{V}_2+\mathcal{V}_3$:
	\begin{subequations}
		\begin{align}
			& \frac{d\mathcal{V}_2}{dt}+\frac{d\mathcal{V}_3}{dt} \nonumber \\
			& \quad\quad = \sum_{j\in\mathcal{N}}M_j(\omega_j-\omega_j^*)\frac{d\omega_j}{dt} \nonumber \\
			& \quad\quad\quad + \sum_{(i,j)\in\mathcal{E}}Y_{ij}(\sin \theta_{ij} - \sin \theta^*_{ij})\left(\frac{d\theta_i}{dt}-\frac{d\theta_j}{dt}\right) \\
			& \quad\quad = \sum_{j\in\mathcal{N}}(\omega_j-\omega_j^*)\left(-D_j\omega_j+p_j-d_j\fixes{-\frac{1}{R_j}\omega_j}-F_j(\boldsymbol{\theta})\right) \nonumber \\
			& \quad\quad\quad + \sum_{(i,j)\in\mathcal{E}}Y_{ij}(\sin \theta_{ij} - \sin \theta^*_{ij})(\omega_i-\omega_j) \\
			& \quad\quad = \sum_{j\in\mathcal{N}}(\omega_j-\omega_j^*)\left(-D_j\omega_j+p_j-d_j\fixes{-\frac{1}{R_j}\omega_j}-F_j(\boldsymbol{\theta})\right)\nonumber\\
			& \quad\quad\quad + \sum_{j\in\mathcal{N}}(\omega_j-\omega_j^*)\left(D_j\omega_j^*\fixes{+\frac{1}{R_j}\omega^*_j}-D_j\omega_j^*\fixes{-\frac{1}{R_j}\omega^*_j}\right) \nonumber \\
			& \quad\quad\quad + \sum_{(i,j)\in\mathcal{E}}Y_{ij}(\sin \theta_{ij} - \sin \theta^*_{ij})(\omega_i-\omega_j) \\
			& \quad\quad = \sum_{j\in\mathcal{N}}-\left(D_j\fixes{+\frac{1}{R_j}}\right)(\omega_j-\omega_j^*)^2 \\
			& \quad\quad\quad + \sum_{j\in\mathcal{N}} (\omega_j-\omega_j^*)\left(p_j-d_j-D_j\omega_j^*\fixes{-\frac{1}{R_j}\omega^*_j}\right)\nonumber\\
			& \quad\quad\quad -\sum_{j\in\mathcal{N}}(\omega_j-\omega_j^*)F_j(\boldsymbol{\theta})\nonumber\\
			& \quad\quad\quad + \sum_{(i,j)\in\mathcal{E}}Y_{ij}(\sin \theta_{ij} - \sin \theta^*_{ij})(\omega_i-\omega_j) \label{eq:stab_V23dot_2}
		\end{align}
	\end{subequations}
	The second equality comes from substituting in \eqref{eq:FreqDev} \eqref{eq:Swing_full2} for their associated derivatives and plugging in \eqref{eq:pj_def}.
	The third equality comes from adding $0=\sum_{j\in\mathcal{N}}(\omega_j-\omega_j^*)\left(D_j\omega_j^*\fixes{+\frac{1}{R_j}\omega^*_j}-D_j\omega_j^*\fixes{-\frac{1}{R_j}\omega^*_j}\right)$.
	The last equality comes from combining $-D_j\omega_j\fixes{-\frac{1}{R_j}\omega_j}$ in the first term and $D_j\omega^*_j\fixes{+\frac{1}{R_j}\omega^*_j}$ in the second term and factoring out $-(\omega_j-\omega^*_j)$, and distributing out $(\omega_j-\omega^*_j)$ to $-F_j(\boldsymbol{\theta})$.
	
	The last term of \eqref{eq:stab_V23dot_2} can be rearranged:
	\begin{subequations}
		\begin{align}
			&\sum_{(i,j)\in\mathcal{E}}Y_{ij}(\sin \theta_{ij} - \sin \theta^*_{ij})(\omega_i-\omega_j) \\
			& \quad = \sum_{(i,j)\in\mathcal{E}}Y_{ij}(\sin \theta_{ij} - \sin \theta^*_{ij})\omega_i \nonumber\\
			& \quad\quad -\sum_{(i,j)\in\mathcal{E}}Y_{ij}(\sin \theta_{ij} - \sin \theta^*_{ij})\omega_j \\
			& \quad = \sum_{i\in\mathcal{N}}\omega_i\sum_{j:i\rightarrow j}Y_{ij}(\sin \theta_{ij} - \sin \theta^*_{ij})\nonumber\\
			& \quad\quad -\sum_{j\in\mathcal{N}}\omega_j\sum_{i:i\rightarrow j}Y_{ij}(\sin \theta_{ij} - \sin \theta^*_{ij}) \\
			& \quad = \sum_{j\in\mathcal{N}}\omega_j\sum_{k:j\rightarrow k}Y_{jk}(\sin \theta_{jk} - \sin \theta^*_{jk})\nonumber\\
			& \quad\quad -\sum_{j\in\mathcal{N}}\omega_j\sum_{i:i\rightarrow j}Y_{ij}(\sin \theta_{ij} - \sin \theta^*_{ij})\\
			& \quad = \sum_{j\in\mathcal{N}}\omega_j\left(\sum_{k:j\rightarrow k}Y_{jk}\sin \theta_{jk}-\sum_{i:i\rightarrow j}Y_{ij}\sin \theta_{ij}\right) \nonumber \\
			& \quad\quad - \sum_{j\in\mathcal{N}}\omega_j\left(\sum_{k:j\rightarrow k}Y_{jk}\sin \theta^*_{jk}-\sum_{i:i\rightarrow j}Y_{ij}\sin \theta^*_{ij}\right)\\
			& \quad = \sum_{j\in\mathcal{N}}\omega_j(F_j(\boldsymbol{\theta})-F_j(\boldsymbol{\theta}^*)) \label{eq:simplify_sin23}
		\end{align}
	\end{subequations}
	The first equality comes from distributing out $Y_{ij}(\sin \theta_{ij} - \sin \theta^*_{ij})$.
	The second equality comes from splitting the each summation into two summations for the origin and destination of the directed edges and factoring out the appropriate $\omega_{\cdot}$.
	The third equality comes from relabeling $j$ to $k$ and $i$ to $j$ in the first term.
	The fourth equality comes from grouping equilibrium terms together and trajectory terms together with respect to $\theta_{\cdot\cdot}$.
	The last equality comes from substituting the definition of net power flow \eqref{eq:NetFlow}.
	
	Plug \eqref{eq:simplify_sin23} back into \eqref{eq:stab_V23dot_2} we get:
	\begin{subequations}
		\begin{align}
			& \frac{d\mathcal{V}_2}{dt}+\frac{d\mathcal{V}_3}{dt} \nonumber \\
			& \quad = \sum_{j\in\mathcal{N}}-\left(D_j\fixes{+\frac{1}{R_j}}\right)(\omega_j-\omega_j^*)^2 \nonumber \\
			& \quad\quad + \sum_{j\in\mathcal{N}} (\omega_j-\omega_j^*)\left(p_j-d_j-D_j\omega_j^*\fixes{-\frac{1}{R_j}\omega^*_j}\right)\nonumber\\
			& \quad\quad -\sum_{j\in\mathcal{N}}(\omega_j-\omega_j^*)F_j(\boldsymbol{\theta})\nonumber\\
			& \quad\quad +\sum_{j\in\mathcal{N}}\omega_j(F_j(\boldsymbol{\theta})-F_j(\boldsymbol{\theta}^*))\\
			& \quad = \sum_{j\in\mathcal{N}}-\left(D_j\fixes{+\frac{1}{R_j}}\right)(\omega_j-\omega_j^*)^2 \nonumber \\
			& \quad\quad +\sum_{j\in\mathcal{N}} (\omega_j-\omega_j^*)\left(p_j-d_j-D_j\omega_j^*\fixes{-\frac{1}{R_j}\omega^*_j}\right) \nonumber\\
			& \quad\quad -\sum_{j\in\mathcal{N}}(\omega_j-\omega_j^*)F_j(\boldsymbol{\theta})\nonumber\\
			& \quad\quad +\sum_{j\in\mathcal{N}}(\omega_j-\omega_j^*+\omega_j^*)(F_j(\boldsymbol{\theta})-F_j(\boldsymbol{\theta}^*)) \\
			& \quad = \sum_{j\in\mathcal{N}}-\left(D_j\fixes{+\frac{1}{R_j}}\right)(\omega_j-\omega_j^*)^2\nonumber\\
			& \quad\quad+\sum_{j\in\mathcal{N}} (\omega_j-\omega_j^*)\left(p_j-d_j-D_j\omega_j^*\fixes{-\frac{1}{R_j}\omega^*_j}-F_j(\boldsymbol{\theta}^*)\right) \nonumber \\
			& \quad\quad +\sum_{j\in\mathcal{N}}\omega_j^*(F_j(\boldsymbol{\theta})-F_j(\boldsymbol{\theta}^*))\\
			& \quad = \sum_{j\in\mathcal{N}}-\left(D_j\fixes{+\frac{1}{R_j}}\right)(\omega_j-\omega_j^*)^2 \nonumber\\
			& \quad\quad +\sum_{j\in\mathcal{N}}(\omega_j-\omega_j^*)\left(p_j-d_j-D_j\omega_j^*\fixes{-\frac{1}{R_j}\omega^*_j}-F_j(\boldsymbol{\theta}^*)\right) \label{eq:stab_V23dot_3}
		\end{align}
	\end{subequations}
	The second equality comes from adding $0=\omega^*_j-\omega^*_j$ inside the first factor of the last term.
	The third equality comes from factoring separating out $(\omega_j-\omega_j^*)F_j(\boldsymbol{\theta})$ from the last term and canceling it out with the third term, and putting $-(\omega_j-\omega_j^*)F_j(\boldsymbol{\theta})$ from the last term into the second term.
	The last equality comes from the last term being equal to zero.
	This is because at equilibrium the condition \eqref{eq:eqlbrm_eq_omega} makes $\sum_{j\in\mathcal{N}}\omega_j^*(F_j(\boldsymbol{\theta})-F_j(\boldsymbol{\theta}^*))=\omega^*\sum_{j\in\mathcal{N}}(F_j(\boldsymbol{\theta})-F_j(\boldsymbol{\theta}^*))$ and since $\sum_{j\in\mathcal{N}}F_j(\boldsymbol{\theta})=\sum_{j\in\mathcal{N}}F_j(\boldsymbol{\theta}^*)=0$ from the balance of power.
	
	At equilibrium with the condition \eqref{eq:eqlbrm_omega_constant} applied to \eqref{eq:Swing_full2} with \eqref{eq:pj_def} gives us $F_j(\boldsymbol{\theta}^*)=p_j-d^*_j-D_j\omega^*_j\fixes{-\frac{1}{R_j}\omega^*_j}$ which when substituted in \eqref{eq:stab_V23dot_3}:
	\begin{align}
		\frac{d\mathcal{V}_2}{dt}+\frac{d\mathcal{V}_3}{dt} & = \sum_{j\in\mathcal{N}}-\left(D_j\fixes{+\frac{1}{R_j}}\right)(\omega_j-\omega_j^*)^2 \nonumber \\
		& \quad+ \sum_{j\in\mathcal{N}}(\omega_j-\omega_j^*)(d^*_j-d_j).
	\end{align}
	By applying the function \eqref{eq:control_dj} to the above equation and factoring out a negative we get \eqref{eq:stab_dV23dt}.
\end{proof}
}

\subsection{Proof of Theorem \ref{thm:stability}}
\label{sec:appendix_stability_proof}

\begin{proof}
	Using Lyapunov's method of stability we start with the following energy function of the trajectory $(\boldsymbol{\theta},\boldsymbol{\omega},s,\mu)$:
	\begin{subequations}
		\begin{align}
			\mathcal{V} & = \mathcal{V}_1 + \mathcal{V}_2 + \mathcal{V}_3 \label{eq:stab_V} \\
			\mathcal{V}_1 & = \frac{1}{2\beta}(\mu-\mu^*)^2  \label{eq:stab_V1} \\
			\mathcal{V}_2 & = \frac{1}{2}\sum_{j\in\mathcal{N}}M_j(\omega_j-\omega_j^*)^2 \label{eq:stab_V2} \\
			\mathcal{V}_3 &  = \sum_{(i,j)\in\mathcal{E}}\int_{\theta^*_{ij}}^{\theta_{ij}}Y_{ij}(\sin u - \sin \theta^*_{ij})du \label{eq:stab_V3}
		\end{align}
	\end{subequations}
	where $\theta_{ij}:=\theta_i-\theta_j$ is the difference in phase angle deviations along each line $(i,j)\in\mathcal{E}$ and $(\boldsymbol{\theta}^*,\boldsymbol{\omega}^*,s^*,\mu^*)$ are the equilibrium satisfying Assumption \ref{as:line_angles} defined by Definition \ref{def:eqlbrm2}.
	
	First we must show that $\mathcal{V}$ is nonnegative.
	Since $\mathcal{V}_1$ from \eqref{eq:stab_V1} and $\mathcal{V}_2$ from \eqref{eq:stab_V2} are sums of quadratic functions with nonnegative coefficients, $\mathcal{V}_1\geq 0$ and $\mathcal{V}_2\geq 0$ for any trajectory $\mu$ and $\boldsymbol{\omega}$, respectively.
	Assumption \ref{as:line_angles} makes each of the integrals in \eqref{eq:stab_V3} non-negative in the neighborhood of $\theta^*_{ij}$ and zero only when $\theta_{ij}=\theta^*_{ij}$.
	This results in $\mathcal{V}_3\geq 0$ for any trajectory $\boldsymbol{\theta}$ within the described neighborhood.
	Therefore, the total energy function from \eqref{eq:stab_V} $\mathcal{V}\geq 0$ for any trajectory $(\mu,\boldsymbol{\omega},\boldsymbol{\theta})$ in a neighborhood of the equilibrium point under consideration.
	
	Second we must show that $\frac{d\mathcal{V}}{dt}$ is nonpositive.
	From Lemma \ref{thm:dV1dt} we have:
	\begin{align}
		\frac{d\mathcal{V}_1}{dt} = & \sum_{j\in\mathcal{N}}(a_j\mu-a_j\mu^*)\left(d_j(\omega_j,\mu) - d_j(\omega_j^*,\mu^*) \right) \nonumber \\
		& -(\mu-\mu^*)\left((g')^{-1}(\mu)-(g')^{-1}(\mu^*)\right)  \label{eq:stab_dV1dt}
	\end{align}
	and from Lemma \ref{thm:dV23dt} we have:
	\begin{align}
		\frac{d\mathcal{V}_2}{dt}+\frac{d\mathcal{V}_3}{dt} = & \sum_{j\in\mathcal{N}}-(\omega_j-\omega_j^*)\left(d_j(\omega_j,\mu) - d_j(\omega_j^*,\mu^*) \right) \nonumber \\
		& +\sum_{j\in\mathcal{N}}-\left(D_j\fixes{+\frac{1}{R_j}}\right)(\omega_j-\omega_j^*)^2. \label{eq:stab_dV23dt}
	\end{align}
	Adding them together we get:
	\begin{subequations}
		\begin{align}
			\frac{d\mathcal{V}}{dt} = & \frac{d\mathcal{V}_1}{dt}+\frac{d\mathcal{V}_2}{dt}+\frac{d\mathcal{V}_3}{dt} \\
			= & \sum_{j\in\mathcal{N}}-((\omega_j-a_j\mu)-(\omega_j^*-a_j\mu^*))\left(d_j(\omega_j,\mu) - d_j(\omega_j^*,\mu^*) \right) \nonumber \\
			& -(\mu-\mu^*)\left((g')^{-1}(\mu)-(g')^{-1}(\mu^*)\right) \nonumber \\
			& + \sum_{j\in\mathcal{N}}-\left(D_j\fixes{+\frac{1}{R_j}}\right)(\omega_j-\omega_j^*)^2 \label{eq:lyapunov_deriv} \\
			\leq & 0
		\end{align}
	\end{subequations}
	\fixes{The second equality comes from combining terms associated with $d_j$.
	The inequality comes from analyzing each term individually.
	In the first summation, if $d_j(\omega_j,\mu)\in(\underline{d}_j,\overline{d}_j)$ then the $j$th summand is strictly less than zero when $(\omega_j-a_j\mu)\neq(\omega^*_j-a_j\mu^*)$ and is equal to zero when $(\omega_j-a_j\mu)=(\omega^*_j-a_j\mu^*)$  due to the strict convexity in Assumption \ref{as:convexity} with \eqref{eq:control_dj} making $d_j(\omega_j,\mu)$ an increasing function of $(\omega_j-a_j\mu)$.
	However, if $d_j(\omega_j,\mu)\in\{\underline{d}_j,\overline{d}_j\}$ then the summand is not necessarily strictly less than zero when $(\omega_j-a_j\mu)\neq(\omega^*_j-a_j\mu^*)$ since it could be that $d_j(\omega_j,\mu)=d_j(\omega^*_j,\mu^*)\in\{\underline{d}_j,\overline{d}_j\}$ which in that case the summand is equal to zero.
	The second term is strictly less than zero when $\mu\neq\mu^*$ and is equal to zero when $\mu=\mu^*$ due to the strict convexity in Assumption \ref{as:convexity} making $(g')^{-1}(\mu)$ an increasing function of $\mu$.
	In the final summation, the $j$th summand is less than zero when $\omega_j\neq\omega^*_j$ and equal to zero when $\omega_j=\omega^*_j$ since it is a quadratic function multiplied by a negative constant.
	In fact, because of the strict monotonicity of the second term with $\mu$ and the $j$th summand in the last summation with $\omega_j$, only at the equilibrium (without specifying $\boldsymbol{\theta}^*$) does $\frac{d\mathcal{V}}{dt}=0$.
	As consequence of Theorem \ref{thm:optimality}, the control laws satisfy KKT condition \eqref{eq:gfc_kkt_stat3} at equilibrium.
	Therefore, for each bus $j$, $\omega^*_j$ must be equal to a single global value $\omega^*$ in order for \eqref{eq:gfc_kkt_stat3} to be satisfied at equilibrium.
	Since the $d_j(\omega_j,\mu)$ defined in \eqref{eq:control_dj} is a function of $\omega_j$ and $\mu$, then $d_j\rightarrow d_j^*$ as $\omega_j\rightarrow \omega^*$ and $\mu\rightarrow \mu^*$.
	Since the excess computational power $s$ defined in \eqref{eq:s_of_djs} is a linear function of each $d_j$, then $s\rightarrow s^*$ as $d_j\rightarrow d_j^*:\forall j\in\mathcal{D}$.
	Since the power injection is $P_j$ as defined in \eqref{eq:pj_def} is a function of $\omega_j$ and $d_j$, then $P_j\rightarrow P^*_j$ as $\omega_j\rightarrow\omega^*_j$ and $d_j\rightarrow d_j^*$.}
	This shows that the equilibrium point under consideration $(\boldsymbol{\omega}^*,\mathbf{P}^*,\mathbf{d}^*,s^*,\mu^*)$ is asymptotically stable.
	This completes the Lyapunov method for asymptotic stability.
\end{proof}

\begin{IEEEbiography}[{\includegraphics[width=1in,height=1.25in,clip,keepaspectratio]{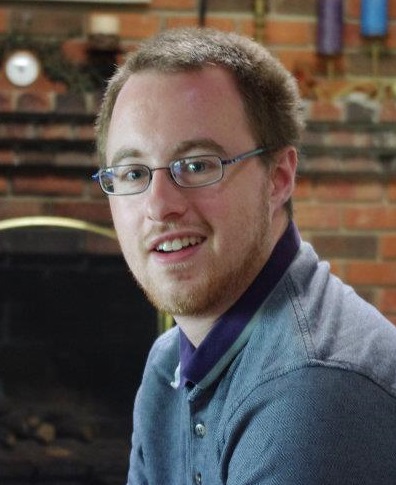}}]{Joshua Comden}
	is currently a Ph.D. Candidate at Stony Brook University studying Operations Research in the Department of Applied Mathematics and Statistics and was awarded the STRIDE Fellowship from 2017 to 2018.
	He received his M.S. in Applied Mathematics and Statistics in 2015 from Stony Brook University and his B.E. in Chemical Engineering from University of Delaware in 2009.
	His research interests are in the design and analysis of demand response programs and continual learning algorithms.
\end{IEEEbiography}

\begin{IEEEbiography}[{\includegraphics[width=1in,height=1.25in,clip,keepaspectratio]{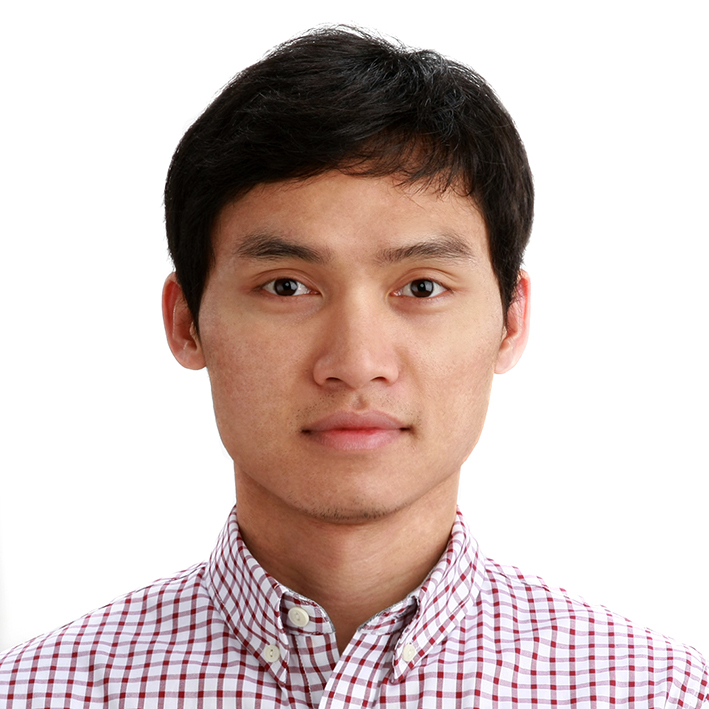}}]{Tan N. Le}
	is a Computer Science Ph.D. candidate at Stony Brook University (NY, USA) \& SUNY Korea.
	He completed his M.E. of Information Telecommunication at Soongsil University in 2011 and B.E. of Information Technology (2008) at PTITHCM, Vietnam.
	He got the ITCCP fellowship from 2013 to 2015.
	His current research interests are demand response and resource allocation for big data systems. 
\end{IEEEbiography}

\begin{IEEEbiography}[{\includegraphics[width=1in,height=1.25in,clip,keepaspectratio]{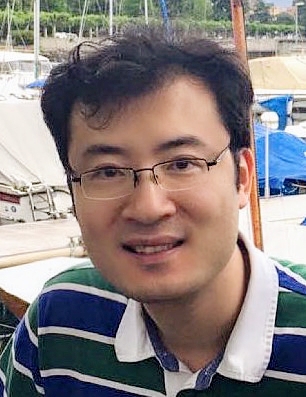}}]{Yue Zhao}
	(S’06, M’11) is an Assistant Professor of Electrical and Computer Engineering at Stony Brook University. He received the B.E. degree in Electronic Engineering from Tsinghua University, Beijing, China in 2006, and the M.S. and Ph.D. degrees in Electrical Engineering from the University of California, Los Angeles (UCLA), Los Angeles in 2007 and 2011, respectively. His current research interests include statistical signal processing, machine learning, optimization, game theory, and their applications in smart grid, renewable energy integration, and infrastructure networks. 
\end{IEEEbiography}

\begin{IEEEbiography}[{\includegraphics[width=1in,height=1.25in,clip,keepaspectratio]{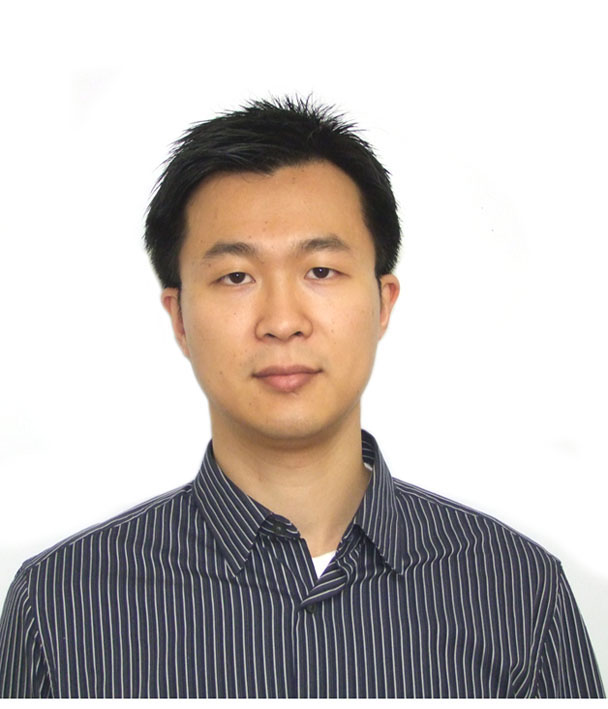}}]{Bong Jun Choi}
	received his B.Sc. and M.Sc. degrees from Yonsei University, Korea, both in electrical and electronics engineering, and the Ph.D. degree from University of Waterloo, Canada, in electrical and computer engineering.
	He is currently an assistant professor at the School of Computer Science and Engineering \& School of Electronic Engineering, Soongsil University, Seoul, Korea.
	His current research focuses on energy efficient networks, distributed mobile wireless networks, smart grid communications, and network security.
\end{IEEEbiography}

\begin{IEEEbiography}[{\includegraphics[width=1in,height=1.25in,clip,keepaspectratio]{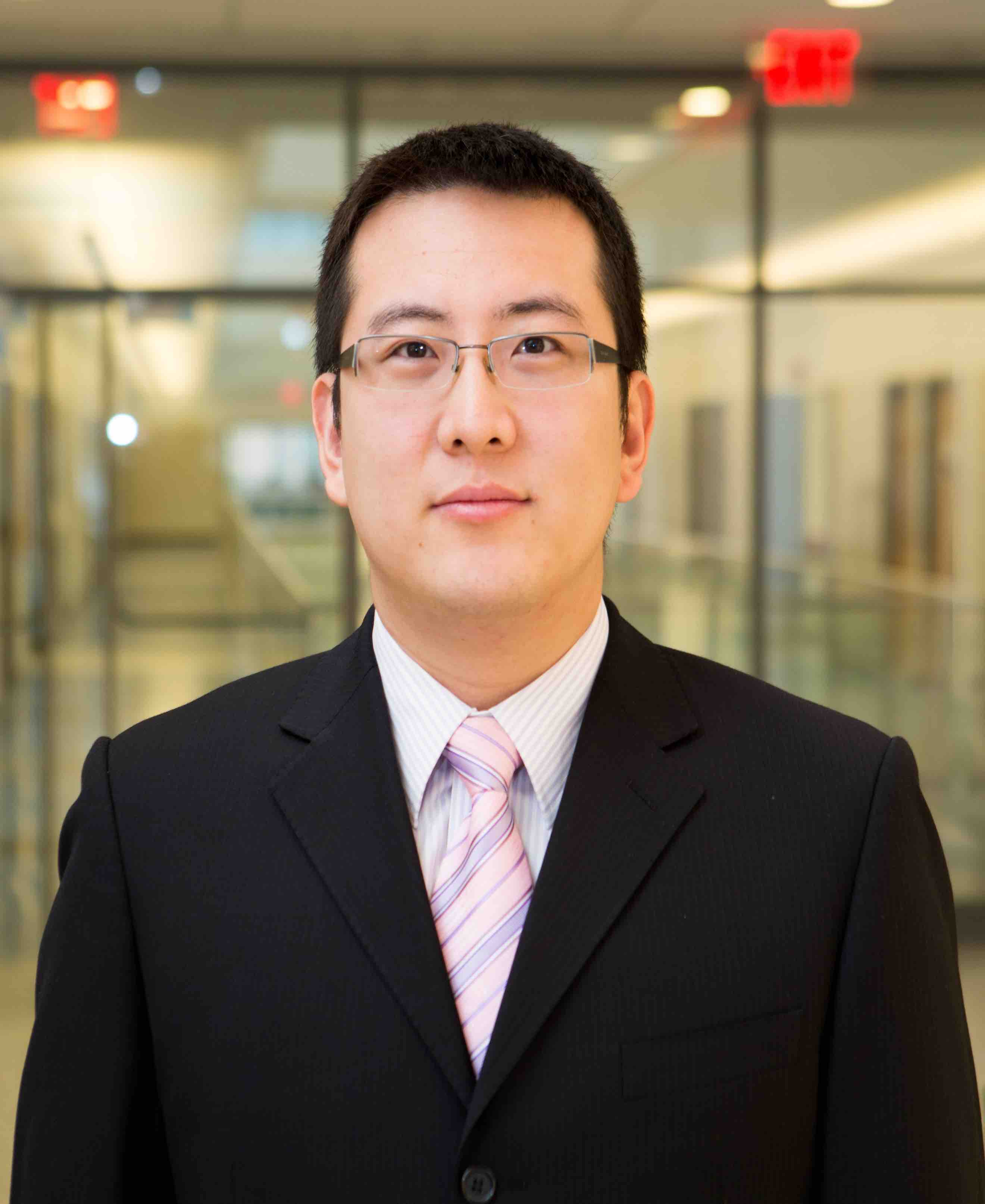}}]{Zhenhua Liu}
	(S08, M13) is an Assistant Professor in the Department of Applied Mathematics and Statistics and Computer Science at Stony Brook University, Stony Brook, NY, USA.
	He received a B.E. degree in Measurement and control in 2006, an M.S. degree in Computer Science in 2009, both from Tsinghua University, and a Ph.D. degree in Computer Science at the California Institute of Technology, Pasadena, CA, USA, in 2014.
	His current research lies in the intersection of energy and big data systems.
\end{IEEEbiography}

\end{document}